\documentclass[11pt, oneside]{article}   	% use "amsart" instead of "article" for AMSLaTeX format
\usepackage[top=2cm, bottom=2cm, left=2cm, right=2cm]{geometry}                		% See geometry.pdf to learn the layout options. There are lots.
\usepackage{dsfont}
\usepackage{graphicx}				% Use pdf, png, jpg, or eps with pdflatex; use eps in DVI mode
\usepackage{epstopdf}
\usepackage{placeins}
\DeclareGraphicsExtensions{.eps}	
\usepackage{amssymb,amsmath,amsthm}% mathtools}
\usepackage{amsfonts}
\usepackage{algorithm}
\usepackage{algorithmic}
\usepackage{caption}
\usepackage{subcaption}
\usepackage{color}
\usepackage{centernot}
\usepackage{bbm}

\newtheorem{theorem}{Theorem}
\numberwithin{theorem}{section}
\newtheorem{lemma}[theorem]{Lemma}

\newtheorem{corollary}[theorem]{Corollary}
\newtheorem{proposition}[theorem]{Proposition}

\newtheorem{definition}[theorem]{Definition}
\numberwithin{equation}{section}

\usepackage{url}

\newcommand*\samethanks[1][\value{footnote}]{\footnotemark[#1]}

\title{
Crossing estimates from metric graph and discrete GFF
}
\author{Jian Ding \thanks{Partially supported by an NSF grant DMS-1757479.}  \\ University of Pennsylvania
\\dingjian@wharton.upenn.edu
\and
Mateo Wirth \samethanks \\ University of Pennsylvania\\mwirth@wharton.upenn.edu
\and
Hao Wu \thanks{Supported by Chinese Thousand Talents Plan for Young Professionals and Beijing Natural Science Foundation (Z180003). } \\ Tsinghua University \\hao.wu.proba@gmail.com}
\date{\today}

\begin{document}

\thispagestyle{empty}
\maketitle
\begin{abstract}
We compare level-set percolation for Gaussian free fields (GFFs) defined on a rectangular subset of $\delta \mathbb{Z}^2$ to level-set percolation for GFFs defined on the corresponding metric graph as the mesh size $\delta$ goes to 0. In particular, we look at the probability that there is a path that crosses the rectangle in the horizontal direction on which the field is positive. We show this probability is strictly larger in the discrete graph. In the metric graph case, we show that for appropriate boundary conditions the probability that there exists a closed pivotal edge for the horizontal crossing event decays logarithmically in $\delta$. In the discrete graph case, we compute the limit of the probability of a horizontal crossing for appropriate boundary conditions. \\
\noindent\textbf{Keywords.} Gaussian free field, crossing event, scaling limit, pivotal edge. 
\end{abstract}
\linespread{1.2}

%\tableofcontents
\newcommand{\diag}{\operatorname{diag}}
\newcommand{\dist}{\operatorname{dist}}
\newcommand{\var}{\operatorname{Var}}
\newcommand{\cov}{\operatorname{Cov}}
\newcommand{\Hm}{\operatorname{Hm}}
\newcommand{\hm}{\operatorname{hm}}
\newcommand{\sign}{\operatorname{sign}}
\newcommand{\adj}{\operatorname{adj}}
\newcommand{\mc}[1]{\mathcal{#1}}
\newcommand{\mbf}[1]{\mathbf{#1}}
\newcommand{\mbb}[1]{\mathbb{#1}}
\newcommand{\til}[1]{\tilde{#1}}
\newcommand{\wtil}[1]{\widetilde{#1}}
\newcommand{\E}{{\mathbb{E}}}
\newcommand{\ctG}{\til{\mc{G}}}
\newcommand{\cH}{\mathcal{H}}
\newcommand{\cI}{\mathcal{I}}
\newcommand{\PP}{{\mathbb{P}}}
\newcommand{\ZZ}{{\mathbb{Z}}}
\newcommand{\bX}{{\mathnormal{\bar{X}}}}
\newcommand{\by}{{\mathnormal{\bar{y}}}}
\newcommand{\bPhi}{\bar{\Phi}}
\newcommand{\one}{\mathbf{\mathbbm{1}}}
\newcommand{\0}{\mathbf{0}}
\newcommand\numberthis{\addtocounter{equation}{1}\tag{\theequation}}
\newcommand{\A}{\mathbb A}
\newcommand{\ABP}[2]{\protect\rotatebox[origin=c]{180}{$#1\A$}}
\newcommand{\AB}{{\mathpalette\ABP\relax}}
\newcommand{\cK}{\mathcal K}
\newcommand{\tB}{\til B}

\newcommand{\R}{\mathbb{R}}
\newcommand{\Z}{\mathbb{Z}}
\newcommand{\N}{\mathbb{N}}
\newcommand{\SLE}{\text{SLE}}
\newcommand{\HH}{\mathbb{H}}
\newcommand{\eps}{\epsilon}
\newcommand{\LF}{\mathcal{F}}
\newcommand{\LA}{\mathcal{A}}
\newcommand{\LB}{\mathcal{B}}
\newcommand{\cond}{\,|\,}
%%%%%%%%%%%%%%%%%%%%%%%%%%%%%%
%%%%%%%%%%%%%%%%%%%%%%%%%%%%%%

\section{Introduction}
In this note we will prove several results on level-set percolation of Gaussian free fields on the square lattice $\mbb{Z}^2$. The thrust of our results is that the probability that there exists a crossing of the domain in which the field is defined differs depending on whether we consider a discrete Gaussian free field or a metric graph Gaussian free field. Additionally, in the metric graph case we prove a bound on the probability that there exists a closed pivotal edge for the crossing. We begin with some basic definitions before stating our main results.

We call two vertices $u,v \in \mbb{Z}^2$ adjacent, and write $u \sim v$, if $|u - v| = 1$, where $|\cdot|$ denotes the standard Euclidean norm. Throughout, we will let $V \subset \mbb{Z}^2$ be a proper subset of $\mbb{Z}^2$. For such a set, we let $\partial V = \{v \in V \,:\, \exists u \in V^c, u \sim v\}$ and $V^o = V \setminus \partial V$.
%%%%%%%%%%
We will denote by $\{S_t : t \geq 0\}$ the continuous-time simple random walk on $\ZZ^2$ with expected holding time $1$ at each vertex, by $\mbb{P}_v$ the law of $S$ where $S_0 = v$, and by $\mbb E_v$ the expectation with respect to $\mbb P_v$.  For $V \neq \mbb{Z}^2$ and $f: \partial V \to \mathbb{R}$ we call a Gaussian process $\{\phi(v): v\in V\}$ a discrete Gaussian free field (discrete GFF) on $V$ with boundary condition $f$ if its mean is the harmonic extension of $f$ to $V$ and its covariance is the Green's function on $V$. That is, letting $\zeta = \inf\{t \geq 0 \,:\, S_t \in \partial V\}$,
\begin{align}
\E[\phi(v)] &= \E_v[f(S_\zeta)], \quad v \in V\label{Discrete harmonic extension}\\
\cov[\phi(u) \phi(v)] = G(u,v) &= \E_u\left[ \int_0^{\zeta} \one_{\{S_t = v\}} dt\right], \quad u,v \in V\,.\label{Green definition}
\end{align}
We note that $\phi(v) = f(v)$ for $v \in \partial V$ (since $\var[\phi(v)] = G(v,v) = 0$), and that the distribution of $\phi$ is uniquely determined by~\eqref{Discrete harmonic extension} and~\eqref{Green definition}.  
%%%%%%%%%

Next, we extend these definitions to metric graphs on $\mbb Z^2$. 
To simplify notation, we will identify each subset $V$ of the lattice with the graph with vertex set $V$ and where two vertices $u,v \in V$ are connected by an edge if $u \sim v$. We denote by $E$ the edge set of this graph. To each $e\in E$ we associate a different compact interval $I_e$ of length $2$ and identify the endpoints of this interval with the two vertices adjacent to $e$. The metric graph $\til V$ associated to $V$ is then defined to be $\til V = \cup_{e\in E} I_e$. With these definitions, it was shown in \cite{Lupu16} that the metric graph Gaussian free field (metric graph GFF) on $\til V$ with boundary condition $f$, denoted by $\{\til{\phi}(v) \,:\, v \in \til V\}$, can be constructed by extending $\phi$ to $\til V$ in the following manner: for adjacent vertices $u$ and $v$, the value of $\til{\phi}$ on the edge $e(u,v)$, conditioned on $\phi(u)$ and $\phi(v)$, is given by an independent bridge of length $2$ of a Brownian motion with variance 2 at time 1 with boundary values $\phi(u)$ and $\phi(v)$. 

With these definitions in place, we specify the crossing events that we will study in this paper. It will be convenient to think of $\til{\mbb{Z}}^2$ as a subset of the complex plane $\mbb{C}$. That is, taking $\mbb{Z}^2$ as a subset of $\mbb{C}$ in the obvious way, we identify the interval $I_e$ corresponding to an edge $e = e(u,v)$ in the discrete graph with the line segment between $u$ and $v$ (this requires re-scaling $I_e$).  For $L > 0$, we let $R_L$ be the rectangle 
\[R_L = \{z\,:\, 0< \Re(z) < L,\, 0 < \Im(z) < 1\}.\] 
Given such a rectangle, we let $(a,b,c,d)$ be its corners, listed in counter-clockwise order with $b = 0$. For $\delta > 0$, we let $V_\delta = R_L \cap \delta \mbb Z^2$ and $(a_\delta,b_\delta,c_\delta,d_\delta)$ be the corners of $V_\delta$, listed in counter-clockwise order so that $b_\delta$ is closest to the origin (we always take $\delta \leq (L \wedge 1)/3$ so that the corners are distinct and $V_\delta^o$ is non-empty). For two vertices $u,v \in \partial V_\delta$, we let $[u,v)$ be the counter-clockwise arc from $u$ to $v$ in $\partial V_\delta$ which contains $u$ but not $v$.
We let $\phi_\delta$ be a discrete GFF on $V_\delta$. We say $\phi_\delta$ has zero boundary condition if $\phi_\delta(v) = 0$ for $v \in \partial V_\delta$, and that it has alternating boundary condition if there exists $\lambda >0$ such that
\begin{equation}\label{eqn::alternatingbc}
\phi_{\delta}(v) =
\begin{cases}
	\lambda, & v \in [a_\delta, b_\delta) \cup [c_\delta, d_\delta),\\
	-\lambda, & v \in [b_\delta,c_\delta) \cup [d_\delta, a_\delta).
\end{cases}
\end{equation}

 In either case, we let $\til{\phi}_\delta$ be the corresponding Gaussian free field on the metric graph $\til{V}_\delta$ such that $\til{\phi}_{\delta}(v) = \phi_{\delta}(v)$ for $v \in V_\delta$. Our goal is to study the probability that $\phi_\delta$ (resp. $\til{\phi}_\delta$) gives a positive horizontal crossing of $V_\delta$ (resp. $\til{V}_\delta$). That is, we want to study the event that there exists a path in $V_\delta$ (resp. $\til V_\delta$) from $[a_\delta, b_\delta)$ to $[c_\delta, d_\delta)$ such that $\phi_\delta$ (resp. $\til{\phi}_\delta$) is non-negative on this path. We denote this event by
\[
\left\{[a_\delta, b_\delta) \stackrel{\phi_\delta \geq 0}{\longleftrightarrow} [c_\delta, d_\delta)\right\}
\]
in the discrete case and similarly in the metric graph case. If $\til\phi_\delta$ has zero boundary condition, the probability of a positive crossing as defined above is equal to zero. For $\phi_\delta$ this probability is one since we can just take a path on $\partial V_\delta$. Therefore, in this case we take $[a_\delta', b_\delta') = \{u \in V_\delta^o \,:\, \exists v \in [a_\delta, b_\delta), u \sim v\}$ and similarly for $[c_\delta', d_\delta')$, and consider the event
\[
\left\{[a'_\delta, b'_\delta) \stackrel{\phi_\delta > 0}{\longleftrightarrow} [c'_\delta, d'_\delta)\right\},
\]
and similarly for $\til\phi_\delta$. It is clear that with probability 1, $\phi(v) \neq 0$ for all $v \in V_\delta^o$, and the zero set of $\til\phi_\delta$ has no isolated points. We will assume that these conditions hold throughout.

In the zero-boundary case, we show that the probability that there is a positive crossing in the metric graph decays to 0 as $\delta \to 0$ (and provide a bound on the rate of decay), while the probability that there is a positive crossing in the discrete graph remains bounded away from zero.

\begin{theorem}\label{thm::Zero boundary result}
Let $L >0$ be a constant, $\phi_\delta$ be a zero-boundary discrete GFF on $V_\delta$, and $\til\phi_\delta$ be the corresponding metric graph GFF on $\til V_\delta$. There exist constants $c = c(L) > 0$ and $\delta_0 = \delta_0(L) > 0$ such that for $\delta \leq \delta_0$,
\begin{equation}\label{MGFF zero boundary result}
\mbb{P}\left([a'_\delta, b'_\delta) \stackrel{\til\phi_\delta > 0}{\longleftrightarrow} [c'_\delta, d'_\delta)\right) \leq \frac{c}{\sqrt{|\log(\delta)|}}.
\end{equation}
In contrast there exists $\eta_1 = \eta_1(L) > 0$ such that for $\delta \leq \delta_0$,
\begin{equation}\label{DGFF zero boundary result}
\eta_1 \leq \mbb{P}\left([a'_\delta, b'_\delta) \stackrel{\phi_\delta > 0}{\longleftrightarrow} [c'_\delta, d'_\delta)\right) \leq 1 - \eta_1.
\end{equation}
\end{theorem}

For alternating boundary conditions, we show that the probability that there is a positive crossing in the metric graph remains bounded away from zero as $\delta \to 0$, but is also strictly smaller than the probability that there is a positive crossing in the discrete graph.
\begin{theorem}\label{thm::Alternating boundary result}
Let $L, \lambda >0$ be positive constants, $\phi_\delta$ be a discrete GFF on $V_\delta$ with alternating boundary condition~\eqref{eqn::alternatingbc}, and $\til\phi_\delta$ be the corresponding metric graph GFF on $\til V_\delta$. There exist constants $\eta_2 = \eta_2(L,\lambda)> 0$, $\eta_3 = \eta_3(L,\lambda) > 0$, and $\delta_0 = \delta_0(L) > 0$ such that for $\delta \leq \delta_0$,
\begin{align}
\eta_2 \leq \mbb{P}\left([a_\delta, b_\delta) \stackrel{\til \phi_\delta \geq 0}{\longleftrightarrow}[c_\delta, d_\delta)\right) \leq 1 - \eta_2, \label{MGFF alternating boundary result}\\
\eta_3 \leq \mbb{P}\left([a_\delta, b_\delta) \stackrel{\phi_\delta \geq 0}{\longleftrightarrow}[c_\delta, d_\delta)\right) \leq 1 - \eta_3.\label{DGFF alternating boundary result}
\end{align}
Additionally, there exists $\eta_4 = \eta_4(L,\lambda)>0$ such that for $\delta \leq \delta_0$,
\begin{equation}\label{Alternating boundary difference result}
\mbb{P}\left([a_\delta, b_\delta) \stackrel{\phi_\delta \geq 0}{\longleftrightarrow}[c_\delta, d_\delta)\right) - \mbb{P}\left([a_\delta, b_\delta) \stackrel{\til\phi_\delta \geq 0}{\longleftrightarrow}[c_\delta, d_\delta)\right) \geq \eta_4.
\end{equation}
\end{theorem}

It is then natural to ask what are the limits of the crossing probabilities in~\eqref{DGFF zero boundary result}, \eqref{MGFF alternating boundary result} and~\eqref{DGFF alternating boundary result}. We will give an explicit answer to the one in~\eqref{DGFF alternating boundary result}; for the other cases, see discussion in Section~\ref{sec::scaling_limits}. There exists $\lambda_0=\lambda_0(\Z^2)>0$ such that the level line in discrete GFF with alternating boundary condition~\eqref{eqn::alternatingbc} converges to $\SLE_4(-2;-2)$ process, see Section~\ref{sec::scaling_limits}. Consequently, we have the following. 

\begin{theorem}\label{thm::crossproba_limit}
When $\lambda=\lambda_0$, the crossing probability in~\eqref{DGFF alternating boundary result} has the following limit: 
\begin{equation}\label{eqn::crossingproba_limit}
\lim_{\delta\to 0}\mbb{P}\left([a_\delta, b_\delta) \stackrel{\phi_\delta \geq 0}{\longleftrightarrow}[c_\delta, d_\delta)\right)=
\frac{(\varphi(b)-\varphi(a))(\varphi(d)-\varphi(c))}{(\varphi(c)-\varphi(a))(\varphi(d)-\varphi(b))},
\end{equation}
where $\varphi$ is any conformal map from $R_L$ onto the upper-half plane $\HH$ with $\varphi(a)<\varphi(b)<\varphi(c)<\varphi(d)$.
\end{theorem}

For the metric graph GFF, we can also prove a bound on the probability that there exists a closed pivotal edge for the event that there is a positive crossing. To specify what we mean, we begin by defining an edge percolation model. Let $E_\delta$ be the edge set of the nearest-neighbor graph on $V_\delta$ and $\omega: E_\delta \to \{0,1\}$ be such that $\omega(e) = 1$ if $\til\phi_\delta(u) \geq 0$ for all $u \in I_e$ and $\omega(e) = 0$ otherwise. We say that $e$ is open if $\omega(e) = 1$ and  that $e$ is closed otherwise. For an edge $e \in E_\delta$, we let $\omega^e$ and $\omega_e$ be defined by
\begin{equation}\label{eqn::def_bondperco}
\omega^e(f)  = \begin{cases}
1,\, f = e,\\ \omega(f),\, f \neq e.
\end{cases} \quad
\omega_e(f)  = \begin{cases}
0,\, f = e,\\ \omega(f),\, f \neq e.
\end{cases}
\end{equation}
We write
\[
\left\{[a_\delta, b_\delta) \stackrel{\omega}{\longleftrightarrow}[c_\delta, d_\delta)\right\}
\]
for the event that there is a path of open edges in $\omega$ from $[a_\delta, b_\delta)$ to $[c_\delta, d_\delta)$ (note that this is exactly the event that there is a positive horizontal crossing in $\til V_\delta$). We can then define
\begin{align*}
\{e \text{ is pivotal}\} &= \left\{[a_\delta, b_\delta) \stackrel{\omega^e}{\longleftrightarrow}[c_\delta, d_\delta)\right\} \setminus \left\{[a_\delta, b_\delta) \stackrel{\omega_e}{\longleftrightarrow}[c_\delta, d_\delta)\right\},\\
\{e \text{ is closed pivotal}\} &= \{e \text{ is pivotal}\} \cap \{w(e) = 0\},\\
\{ \text{there exists a closed pivotal edge} \} &= \cup_{e \in E_\delta} \{e \text{ is closed pivotal}\}.
\end{align*}
With these definitions, we obtain the following result
\begin{theorem}\label{thm::Pivotal edges result}
Let $L, \lambda >0$ be positive constants and $\til\phi_\delta$ be the metric graph GFF on $\til V_\delta$ with alternating boundary condition~\eqref{eqn::alternatingbc}.
There exist constants $c = c(L,\lambda) > 0$ and $\delta_0 = \delta_0(L) > 0$ such that for $\delta \leq \delta_0$,
\[
\mbb{P}(\text{\emph{there exists a closed pivotal edge}}) \leq \frac{c}{\sqrt{|\log(\delta)|}}.
\]
\end{theorem}

The rest of the paper is organized as follows. In Section~\ref{Prelims} we introduce the main tools used in the proofs. In Section~\ref{MGFF RSW} we prove the results in Theorems~\ref{thm::Zero boundary result} and~\ref{thm::Alternating boundary result} pertaining to the metric graph GFF $\til\phi_\delta$, and in Section~\ref{DGFF RSW} we prove the results in Theorems~\ref{thm::Zero boundary result} and~\ref{thm::Alternating boundary result} pertaining to the discrete GFF $\phi_\delta$. Assuming the conclusions proved there, we are able to conclude the proof of Theorem~\ref{thm::Zero boundary result} and~\ref{thm::Alternating boundary result} here. Finally, we will discuss the limits of the probabilities in~\eqref{DGFF zero boundary result}, \eqref{MGFF alternating boundary result} and~\eqref{DGFF alternating boundary result}, and prove Theorem~\ref{thm::crossproba_limit} in Section~\ref{sec::scaling_limits}. 

\begin{proof}[Proof of Theorem~\ref{thm::Zero boundary result}]
The bound~\eqref{MGFF zero boundary result} is proved in Section~\ref{subsec::MGFF_zerobc}. The bound~\eqref{DGFF zero boundary result} is proved in Section~\ref{subsec::DGFF_zerobc}. 
\end{proof}
\begin{proof}[Proof of Theorem~\ref{thm::Alternating boundary result}]
Note that it suffices to prove the lower bound in \eqref{MGFF alternating boundary result} to prove \eqref{MGFF alternating boundary result} and \eqref{DGFF alternating boundary result}. This follows from the following considerations. First, since $\phi_\delta$ and $\til\phi_\delta$ coincide on $V_\delta$, the existence of a positive crossing in the metric graph implies the existence of such a crossing in the discrete graph. That is,
\[
\left\{[a_\delta, b_\delta) \stackrel{\til \phi_\delta \geq 0}{\longleftrightarrow}[c_\delta, d_\delta)\right\} \subseteq \left\{[a_\delta, b_\delta) \stackrel{\phi_\delta \geq 0}{\longleftrightarrow}[c_\delta, d_\delta)\right\}.
\]
Therefore, the lower bound in \eqref{MGFF alternating boundary result} implies the lower bound in \eqref{DGFF alternating boundary result}. Next, by symmetry, the lower bound in \eqref{MGFF alternating boundary result} implies that
\[
\mbb{P}\left([b_\delta,c_\delta)\stackrel{\til\phi_\delta \leq 0}{\longleftrightarrow}[d_\delta, a_\delta)\right) \geq \eta_2(L^{-1},\lambda).
\]
The same bound then holds for the probability that there is a strictly negative vertical crossing in the discrete graph (here we use the fact that $\phi_\delta$ is almost surely not equal to zero on $V_\delta$). Finally, since the existence of a strictly negative vertical crossing in the discrete graph implies that there is no positive horizontal crossing (either in the metric graph or the discrete graph), the upper bounds in \eqref{MGFF alternating boundary result} and \eqref{DGFF alternating boundary result} follow. Thus, both~\eqref{MGFF alternating boundary result} and~\eqref{DGFF alternating boundary result} follow from the lower bound in~\eqref{MGFF alternating boundary result} which is proved in Section~\ref{subsec::MGFF_alternatingbc}. 

Finally, the bound in~\eqref{Alternating boundary difference result} is proved in 
Section~\ref{subsec::DGFF_alternatingbc}. 
\end{proof}

\subsection*{Discussions on future directions}
\begin{itemize}
\item The level loops of zero-boundary metric graph GFF will converge to the so-called conformal loop ensemble CLE with $\kappa=4$. From there, we believe that the optimal estimate for the crossing probability in~\eqref{MGFF zero boundary result} is $|\log(\delta)|^{-2}$. However, our method only provides an upper bound $|\log (\delta)|^{-1/2}$. To get the optimal estimate requires more delicate tools. 
\item The level lines of zero-boundary discrete GFF will converge to the so-called branching $\SLE_4(-1;-1)$ process. Then the limit of the crossing probability in~\eqref{DGFF zero boundary result} should be the probability of certain crossing event of $\SLE_4(-1;-1)$. It is an interesting question to derive an explicit formula for such probability. The limit of crossing probabilities in~\eqref{MGFF alternating boundary result} and~\eqref{DGFF alternating boundary result} for general $\lambda$ are also unknown. 
\item We provide an estimate on the existence of pivotal edge in metric graph GFF. It is more important to understand the pivotal vertices in discrete GFF and its connection to continuum GFF. 
\end{itemize}

\noindent\textbf{Acknowledgements.}
We thank J.~Aru, T.~Lupu, A.~Sep\'{u}lveda for helpful discussion on GFF. 

%%%%%%%%%%%%%%%%%%%%%%%%%%%%%%
\section{Preliminaries}
\label{Prelims}

In this section we collect some technical tools and introduce some notation that will be used in the proofs of the main results.

\textit{Notation.} For a real vector $\mathbf x$ (in any dimension), we denote by  $|\mathbf x|$  the Euclidean norm of $\mathbf x$, by $|\mathbf x|_{\ell_1}$ its $\ell_1$-norm, and by $|\mathbf x|_{\ell_\infty}$ its $\ell_\infty$-norm. For a finite set $A$, we will also use $|A|$ to denote the cardinality of $A$. The meaning will be clear from context. We use $A^c$ to denote the complement of the set (or event) $A$.

Throughout the proofs we let $c,c'$ be positive constants which only depend on $L$ and $\lambda$ as in the main theorems, and whose value may change each time they appear.

\subsection{Discrete GFF and metric graph GFF}
\label{subsec::pre_gff}
We begin with an alternate construction of the metric graph GFF which more clearly shows it is a natural analog to the discrete GFF. Namely, one can construct $\til{\phi}$ by first defining a Brownian motion $\{\til{B}_t \,:\, t \geq 0\}$ on $\til{\mbb{Z}}^2$ as in \cite[Section~2]{Lupu16}:  $\til{B}$ behaves like a standard Brownian motion in the interior of the edges, while on the vertices it chooses to do excursions on each incoming edge uniformly at random. For a subset $V \subset \mbb{Z}^2$, we let $\til\zeta = \inf\{ t \geq 0 \,:\, \til{B}_t \in \partial V\}$, and $\{G(u,v) \,:\, u,v \in \til V\}$ be the density of the 0-potential of $\{\til{B}_t \,:\, 0 \leq t < \til\zeta\}$ (with respect to the Lebesgue measure on $\til V$), where $u$ and $v$ are now arbitrary points in $\til{V}$ (not necessarily vertices). It is shown in \cite{Lupu16} that the trace of $\til{B}$ on $V$ (when parametrized by its local time at the vertices) is exactly the continuous-time simple random walk on $V$, and therefore the definition of $G$ here coincides with the one in~\eqref{Green definition} for $u,v \in V$, justifying the abuse of notation. Due to this relation between the metric graph Brownian motion and the simple random walk, we will, by a slight abuse of notation, let $\mbb{P}_v$ be the law of $\til{B}_t$ where $\til B_0 = v$ and $\E_v$ be the expectation with respect to $\mbb P_v$. It was also shown in \cite{Lupu16} that the value of $G$ on $\til{V} \setminus V$ can be obtained by interpolation from the value on $V$. For two pairs of adjacent vertices $(u_1, v_1)$ and $(u_2, v_2)$ in $V$, and two points $w_1$ and $w_2$ on the corresponding edges, taking the convention that either the edges are distinct or $(u_1,v_1) = (u_2,v_2)$ and letting $r_1 = |w_1 - u_1|$ and $r_2 = |w_2 - u_2|$, we have (c.f. \cite[Equation~(2.1)]{Lupu16})
\begin{align}\label{Metric Green Interpolation}
\begin{split}
G(w_1, w_2) = &(1- r_1)(1-r_2) G(u_1,u_2) + r_1 r_2 G(v_1, v_2) + (1-r_1)r_2G(u_1,v_2) \\
&+ r_1(1-r_2)G(v_1, u_2) + 4(r_1\wedge r_2 - r_1r_2) \one_{(u_1,v_1) = (u_2, v_2)}. 
\end{split}
\end{align}

We call a process $\{\til{\phi}(v) \,:\, v \in \til V\}$ a Gaussian free field on $\til V$ if it is a continuous Gaussian process such that there exists a function $f : \partial V \to \mbb{R}$ for which \[\E[\til\phi(v)] = \til\E_v[f(\til B_{\til\zeta})],\quad \cov[\til{\phi}(u) \til{\phi}(v)] = G(u,v).\] 
We note that this definition extends to any compact, connected subset $K \subset \til{\mbb{Z}}^2$.

We remark on a consequence of the construction above. For any finite subset $A \subset \til{\mbb Z}^2$, we can define an electric network with vertex set $W = \mbb{Z}^2 \cup A$, where any two vertices $u,v \in W$ are connected with an edge of conductance $(4|u-v|)^{-1}$ if there is a continuous path in $\til{\mbb Z}^2$ from $u$ to $v$ which does not contain any other points in $W$ (note this path is always contained in $I_e$ for some edge in standard nearest-neighbor graph). Then, if $V \subset W$ is a connected (proper) subset and $\til\phi$ is a metric graph GFF on $\til{V}$, $\{\til\phi(v) \,:\, v \in V\}$ is a discrete GFF on (the electric network with vertex set) $V$.

Next, we state the FKG inequality for the GFF, which will be used repeatedly throughout the paper. Let $\til\phi$ be a metric graph GFF and $\mc F$ be the $\sigma$-field generated by $\til\phi$. An event $A \in \mc{F}$ is increasing if $\one_A$ is increasing as a function of $\til\phi$. For two increasing events $A$ and $B$, we have \cite{Pitt82}
\begin{equation}\label{FKG}
\mbb{P}(A \cap B) \geq \mbb{P}(A) \mbb{P}(B).
\end{equation}
By symmetry, the same holds for decreasing events.

\subsection{Excursion sets and first passage sets}
Next, we introduce certain random subsets of $V_\delta$ and $\til V_\delta$ that are related to the existence of positive crossings. Let $E^{\geq 0}_{\delta}$ (resp. $\til{E}^{\geq 0}_{\delta}$) be the excursion set of $\phi_\delta$ (resp. $\til{\phi}_{\delta}$) above 0. That is,
\[
E^{\geq 0}_{\delta} = \{v \in V_\delta \,:\, \phi_{\delta}(v) \geq 0\},
\]
and similarly for $\til{E}^{\geq 0}_\delta$. The excursion set below zero, $E^{\leq 0}_{\delta}$ (resp. $\til{E}^{\leq 0}_{\delta}$), is defined similarly. Next, we introduce the first passage set of $\phi_\delta$ (resp. $\til{\phi}_{\delta}$) above 0. This set, which we denote by $\A_{\delta,0}$ (resp. $\til{\A}_{\delta,0}$) is the union of all connected components of $E^{\geq 0}_\delta$ that intersect the boundary. That is, 
\[
\A_{\delta,0} = \{v \in V_\delta \,:\,  \exists ~ \text{nearest-neighbor path $\gamma$ from $v$ to $\partial V_\delta$ such that $\phi_\delta \geq 0$ on $\gamma$}\}.
\]
We will denote by $\A_{\delta,0}^l$ the union of all connected components of $\A_{\delta,0}$ intersecting $[a_\delta,b_\delta)$. That is, the left part of the first passage set. Similarly, we denote by $\A_{\delta,0}^r$ the union of all connected components of $\A_{\delta,0}$ intersecting $[c_\delta,d_\delta)$ (the right part of the first passage set). Note that there is a positive horizontal crossing of $R_L$ when $\A_{\delta,0}^l = \A_{\delta,0}^r$. We denote by $\AB_{\delta,0}$ the first passage set \emph{below} 0, by $\AB_{\delta,0}^b$ the union of all connected components of $\AB_{\delta,0}$ intersecting $[b_\delta,c_\delta)$ (the bottom part of the first passage set), and by $\AB_{\delta,0}^t$ the union of all connected components of $\AB_{\delta,0}$ intersecting $[d_\delta,a_\delta)$ (the top part of the first passage set). The metric graph first passage sets $\til{\A}_{\delta,0}$, $\til\A_{\delta,0}^l$, $\til\A_{\delta,0}^r$, and so on are defined similarly.

\subsection{Exploration martingales}
\label{subsec::pre_explorationmart}
In this section we introduce a family of martingales which form the basis of the proofs of our results. We begin with some basic definitions and some fundamental results. For a subset $W \subset \til V_\delta$, we let $\mathcal F_W$ be the sigma algebra generated by $\{\til\phi_\delta(w) \,:\, w \in W\}$.
\begin{definition}[Optional set]\label{Optional set}
Let $\mathcal K$ be a random compact subset of $\til V_\delta$ that almost surely has finitely many connected components. We say $\mathcal K$ is an optional set for $\til\phi_\delta$ if for every deterministic, open subset $W$ of $\til V_\delta$, $\{\mathcal K \subset W \}\in \mathcal{F}_W$. Similarly, we say $\mc{K}$ is optional for $\phi_\delta$ if the following holds. For any deterministic open subset $W$ of $\til V_\delta$, $\{\mc{K} \subset W\} \in \mc{F}_{W \cap V_\delta}$.
\end{definition}
We note that in the definition above, we use the subspace topology on $\til V_\delta$ considered as a subspace of $\til{\mbb{Z}}^2$.

For an optional set $\mc{K}$, we define its $\sigma$-field $\mc{F}_{\mc K}$ by
\[
\mc{F}_{\mc K} = \left\{A \in \mc{F}_{\til V_\delta} \,:\, A \cap \{\mc{K} \subset W\} \in \mc{F}_W, \text{ for all deterministic, open $W \subset \til V_\delta$}\right\}.
\]
Sometimes, it will be useful to consider the field $\sign(\til\phi) = \{\sign(\til\phi(v)) \,:\, v \in \til V_\delta\}$, where $\sign$ is the function
\[
\sign(x) = \begin{cases}
	-1, &x < 0,\\
	0, &x = 0,\\
	1, &x > 0.
	\end{cases}
\]
For a subset $W \subset \til V_\delta$, we let $\mc G_W$ be the $\sigma$-field generated by $\{\sign(\til\phi(w)) \,:\, w \in W\}$. We extend this notation to optional sets in the obvious way.

The strong Markov property of the metric graph GFF, stated below, will allow us to perform detailed analysis of the exploration martingales and is thus fundamental to our proofs.
\begin{theorem}[Strong Markov property, \cite{Lupu16}]\label{Strong Markov property}
Let $\mc{K}$ be optional for $\til\phi_\delta$. Given $\mc{F}_{\mc K}$,  the process $\{\til\phi_\delta(v) \,:\, v \in \overline{\til V_\delta \setminus \mc{K}}\}$ is a metric graph GFF on $\overline{\til V_\delta \setminus \mc{K}}$ with boundary condition given by the restriction of $\til\phi_\delta$ to $\partial \mc{K}\cup \partial\til{V}_\delta$.
\end{theorem}

We introduce some notation for harmonic extension. Assume the same setup as in Theorem~\ref{Strong Markov property}, and let $\tilde{B}$ be a Brownian motion on $\tilde{V}_{\delta}$. Set $\til\zeta=\inf\{t\ge 0: \tilde{B}_t\in\partial \tilde{V}_{\delta}\}$ as in Section~\ref{subsec::pre_gff} and $\tau = \inf\{t \geq 0 \,:\, \til B_t \in \mc{K}\}$ be the hitting time of $\mc{K}$. We define, for $v\in \tilde{V}_{\delta}$ and $w\in \mc{K}\cup \partial\tilde{V}_{\delta}$, 
\[\Hm(v,w;\mc{K}) = \mbb{P}_v\left(\til B_{\tau\wedge \til\zeta} = w\right). \]
Note that the harmonic measure $\Hm(v, \cdot; \mc{K})$ is always supported on a finite set of points since $|\partial \mc{K}| < \infty$ for compact $\mc{K}$ with finitely many connected components. 
With this notation, in Theorem~\ref{Strong Markov property}, we have in fact, for $u,v \in \til V_\delta$,
\begin{align*}
\E[\til\phi_\delta(v) \mid \mc{F}_{\mc K}]&= \sum_{w \in \til{V}_\delta} \Hm(v,w;\mc{K})\til\phi_\delta(w),\\
\E[\til\phi_\delta(u)\til\phi_\delta(v) \mid \mc F_{\mc K}] &= G(u,v) - \sum_{w \in \mc K}\Hm(u,w; \mc K)G(w,v).
\end{align*}
Note also that Theorem~\ref{Strong Markov property} implies the strong Markov property of the discrete GFF since any random subset $\mc K \subset V_\delta$ which is optional for $\phi_\delta$ is optional for $\til\phi_\delta$.

%To simplify notation, for $u \in \til V_\delta$ and $W \subset \til V_\delta$ a Borel set, we let
%\[
%\Hm(u,W) = \sum_{w \in W} \Hm(u,w;W) = \mbb{P}_u(\tau \leq \til\zeta),
%\]
%where $\tau$ is the hitting time of $W$ and $\til\zeta$ is (as usual) the hitting time of $\partial \til V_\delta$. 
Moreover, we define 
\[\Hm(v, \mc{K})=\sum_{w\in\mc{K}}\Hm(v, w; \mc{K})=\PP_v(\tau\le \tilde{\zeta}); \]
for $U \subset \til V_\delta$ a finite subset, we define
\[
\Hm(U,\mc{K}) = \sum_{u \in U}\Hm(u,\mc{K}).
\]

We can now introduce the exploration martingales. For a finite subset $U \subset V_\delta$, we define the ``observable'' $X_U$ by
\[
X_U = \sum_{v \in U} \til{\phi}_\delta(v).
\]
For $\cI_0$ a deterministic, compact subset of $\til{V}_\delta$ (with finitely many connected components), we will let $M$ be the Doob martingale for $X_U$ as we explore $E^{\geq 0}_\delta$, or $\til{E}^{\geq 0}_\delta$ from $\cI_0$. We will specify whether the exploration happens on the metric graph or discrete graph whenever we use an exploration martingale.

To make this precise, we specify what we mean by exploring the excursion set from $\mc I_0$. We begin with the exploration on $\til{E}^{\geq 0}_\delta$ as it is simpler to explain. Let $\til{D}^{\geq 0}_{\delta}$ be the metric graph distance on $\til{E}^{\geq 0 }_\delta$, and $\cI_t$  be the ball of radius $t$ around $\cI_0$ with respect to $\til{D}^{\geq 0}_{\delta}$. We use the convention that for $u,v \in \til V_\delta$ with $u \notin \til E^{\geq 0}_\delta$ and $u \neq v$, $\til D^{\geq 0}_\delta(u,v) = \infty$ and $\til D^{\geq 0}_\delta(u,u) = 0$. It is easy to show that ${\cI}_t$ is an optional set for each $t \geq 0$ \cite{DingWirth18}.

The exploration corresponding to $E^{\geq 0}_\delta$ is similar in spirit but slightly more cumbersome to describe. In this case we take $\mc{V}_0 \subset V_\delta$ and $\mc I_0$ the metric graph on $\mc V_0$. We begin the exploration by setting $\mc{A}_0 = \mc{V}_0 \cap E^{\geq 0}_\delta$, and $\mc{B}_0 = \mc{V}_0 \cap E^{< 0}_\delta$. For $k \geq 1$ an integer, we let
\begin{align*}
\mc{A}_k &= \{v \in (V_\delta^o \setminus \mc{V}_{k-1}) \cap E^{\geq 0}_\delta \,:\, \exists u \in \mc{A}_{k-1}, \, u \sim v\},\\
\mc{B}_k &= \{v \in (V_\delta^o \setminus \mc{V}_{k-1}) \cap E^{< 0}_\delta \,:\, \exists u \in \mc{A}_{k-1}, \, u \sim v\},\\
\mc{V}_k &= \mc{V}_{k-1} \cup \mc{A}_k \cup \mc{B}_k,
\end{align*}
and $\mc I_k$ be the metric graph on $\mc V_k$ (as with $M$, whether $\mc I_k$ corresponds to the metric graph or discrete graph construction will be clear from context). In words, at each time step we explore all unexplored vertices that are adjacent to an explored vertex on which the field is non-negative. Note that $\mc{I}_k$ is an optional set, and in fact $\mc{I}_{k} \in \mc{G}_{\mc{V}_{k-1}}$. To extend this definition to real $t$, we proceed by interpolation. That is, if we let $\mc{E}_k = \{(u,v)\,:\, u \in \mc{I}_{k-1}, \, v \in \mc{A}_{k} \cup \mc{B}_{k}, \, u \sim v\}$ be the edges between $\mc{I}_{k-1}$ and $\mc{I}_k \setminus \mc{I}_{k-1}$ then for $k - 1 < t < k$ we let
\[
\mc{I}_t = \mc{I}_{k-1} \cup\left(\bigcup_{(u,v) \in \mc{E}_{k}}[u, (k - t)u + (t-k - 1) v] \right).
\]
Note that as before $\mc{I}_t \in \mc{G}_{\mc V_{k-1}}$ and so $\mc I_t$ is an optional set for any $t$.

Explorations on $E^{\leq 0}_\delta$ and $\til{E}^{\leq 0}_\delta$ are defined in the same way. As alluded to above, we take 
\[M_t = \E[X_U \mid \mc F_{\mc I_t}].\] 
It is straightforward to check that $M$ is a continuous martingale \cite{DingWirth18}. We now turn to the quadratic variation of the exploration martingale. %For concreteness, we state the properties below in terms of $M_t$ but they hold for $\til M_t$ by replacing $\mc I_t$ by $\til{\mc I}_t$.

It is straightforward to check that for a Doob martingale, the quadratic variation is equal to the decrease in the conditional variance \cite{DingWirth18}. That is,
\[
\langle M \rangle_t = \var[X_U \mid \mc F_{\mc I_0}] - \var[X_U \mid \mc F_{\mc I_t}].
\]
We denote by $G_t$ the Green's function on $\tilde{V}_{\delta}\setminus \mc I_t$, and write
\begin{equation}\label{eqn::def_Hmt}
\Hm_t(u,v) = \Hm(u,v;\mc I_t).
\end{equation}
We get from Theorem~\ref{Strong Markov property} that 
\begin{align}
\var[X_U \mid \mc{F}_{\cI_t}] &=\sum_{u,u' \in U} G_t(u, u')=\sum_{u,u' \in U} \left[G(u,u') - \sum_{v \in \mc I_t}\Hm_t(u,v)G(v,u')\right].
\end{align}
This gives 
\begin{align*}
\langle M \rangle_t &= \sum_{u,u' \in U} \left[\sum_{v \in \mc I_t} \Hm_t(u,v)G(v,u') - \sum_{v' \in \mc I_0}\Hm_0(u,v')G(v',u')\right],\\
&= \sum_{u,u' \in U} \sum_{v \in \mc I_t} \Hm_t(u,v)G_0(v,u'), \\
&= \sum_{v \in \mc I_t} \Hm_t(U,v)G_0(v,U),\numberthis \label{eq::QV form}
\end{align*}
where $G_0(v,U) = \sum_{u \in U} G_0(v,u)$.

\subsection{Brownian motion tools}
In this section, we recall two facts about continuous martingales and Brownian motion that will be useful throughout the paper. The first is \cite[Theorem~1.7 in Chapter~V]{RevusYor99}, stated below, which is a version of the Dubins-Schwarz theorem for martingales of bounded quadratic variation.
\begin{theorem}\label{Time-change theorem}
Let be $M$ a continuous martingale, $T_t = \inf\{s \,:\, \langle M \rangle_s > t\}$, and $W$ be the following process
\[
W_t = \begin{cases}
	M_{T_t} - M_0 & t < \langle M \rangle_\infty\, ,\\
	M_{\infty} - M_0 & t \geq \langle M \rangle_\infty.
	\end{cases}
\]
Then $W$ is a Brownian motion stopped at $\langle M \rangle_\infty$.
\end{theorem}
When applying this theorem, we will generally denote by $B$ a Brownian motion which satisfies $B_t = M_{T_t} - M_0$ for $t < \langle M \rangle_\infty$ but is not stopped at $\langle M \rangle_\infty$, so that $W_t = B_{t \wedge \langle M \rangle_\infty}$. 
Suppose $M$ is the exploration martingale in Section~\ref{subsec::pre_explorationmart}. By Theorem~\ref{Time-change theorem}, the process $\{M_{T_t}-M_{0} \,:\, t \geq 0\}$ is independent of $\mc{F}_{\mc I_0}$, so we will generally take $B$ to be independent of $\mc{F}_{\mc I_0}$ as well.

The second result gives the distribution of the hitting time of a line by Brownian motion.
\begin{proposition}\label{prop::Hitting times with drift}
Let $B$ be a standard one-dimensional Brownian motion, $\Phi$ be the distribution function of the standard normal distribution, and $\bar{\Phi}(x) = 1 - \Phi(x)$. For $m \in \mathbb{R}$ and $b > 0$, let $\tau = \inf \{ t \geq 0 \,:\, B_t \leq mt - b\}$. Then for $T > 0$,
\[
\PP(\tau \leq T) = \bar{\Phi}\left(\frac{b}{\sqrt{T}} - m\sqrt{T} \right) + e^{2bm}\bar{\Phi}\left(\frac{b}{\sqrt{T}} + m \sqrt{T}\right).
\]
\end{proposition}
\begin{proof}
	The density of $\tau$ is given in \cite[Equation~(2.0.2) in Part II]{BrownianMotionHandbook}. Taking an integral then gives the desired result.
\end{proof}
We note two facts that follow directly from Proposition~\ref{prop::Hitting times with drift} and will be used repeatedly throughout the paper. First, letting $m = 0$ we obtain that $\sup\{B_s \,:\, 0 \leq s \leq t\}$ and $-\inf\{B_s \,:\, 0 \leq s \leq t\}$ have the same distribution as $|B_t|$. Second, for $m < 0$ and $b > 0$ we have by letting $T \to \infty$ that $\PP(\tau < \infty) = e^{2bm} < 1$.

\subsection{Random walk estimates}
We conclude this section with some results about Green's functions and harmonic measures which will be used later. The first result follows from \cite[Theorem~4.4.4, Proposition~4.6.2]{LawlerLimic10}. It will be useful in comparing the Green's function in different domains. Below and throughout the paper, a subset $V \subset \mbb{Z}^2$ is simply connected if for any loop in $V$, all the vertices in the interior of the loop (i.e. separated from infinity by the loop) are contained in $V$.
\begin{lemma}\label{Greens scaling}
There exists a universal constant $C > 0$ such that the following holds. For $V \subset \mbb{Z}^2 $ simply connected, let $G$ be the Green's function on $V$. For $v \in V^o$, let $\Delta = \dist(v,\partial V)$ be the Euclidean distance between $v$ and $\partial V$. Then
\[
\Big|G(v,v) - \frac{\pi}{2} \log\left(\Delta + 1\right)\Big| \leq C.
\]
\end{lemma}

Next, we provide two estimates on harmonic measures for a random walk started near the boundary of a box. The first provides an upper bound on the probability that a random walk started near the left side of a box doesn't exit the box through the left side. Recall that $S$ is a simple random walk on $\mbb{Z}^2$, $\mbb{P}_v$ is the law of $S$ started at $v$, and that we take $\mbb{Z}^2 \subset \mbb{C}$ in the obvious way.
\begin{lemma}\label{lm-Box Hm upper bound}
There exists a constant $c > 0$ such that the following holds. For $N,M \geq 2$, $R = [0,N]\times[-M,M]$, $V = R \cap \mbb Z^2$, and $\zeta = \min\{n \geq 1\,:\, S_n \in \partial V\}$. We have
\[
\mbb{P}_1(\Re(S_\zeta) \neq 0) \leq \frac{c}{N\wedge M}.
\] 
\end{lemma}
\begin{proof}
Let $\zeta' = \min\{n \geq 1 \,:\, \Re(S_n) = 0\}$ and $K = N \wedge M$. We have by \cite[Theorem 8.1.2]{LawlerLimic10} that 
\[
\mbb{P}_1(|\Im(S_{\zeta'})|  > K ) \leq \frac{c}{K}.
\]
Let $D =  \partial V \setminus \{z \,:\, \Re(z) = 0\}$ and note that there exists a universal constant $c' > 0$ such that
\[
\mbb{P}_v(|\Im(S_{\zeta'})| > K) \geq c', \quad \forall v \in D.
\]
Finally, the conclusion follows by noting
\[
\mbb{P}_1(|\Im(S_{\zeta'})|  > K ) = \sum_{v \in D} \mbb{P}_1(S_\zeta = v)\mbb{P}_v(|\Im(S_{\zeta'})| > K) \geq c' \mbb{P}_1(\Re(S_\zeta) \neq 0).
\]
\end{proof}

The second result provides a lower bound on the probability that a random walk started near the left side of a box will exit the box through the right side.
\begin{lemma}\label{lm-Box Hm lower bound}
For any $a > 0$ there exists a constant $c_a > 0$ such that the following holds. For $N, M \geq 2$ with $M \geq a N$, $R = [0,N]\times [-M,M]$, $V = R\cap\mbb{Z}^2$, and $\zeta= \min\{n \geq 1\,:\, S_n\in \partial V\}$, we have
\[
\mbb{P}_{1}(\Re(S_{\zeta}) = N) \geq \frac{c_a}{N}.
\]
\end{lemma}
\begin{proof}
Let $M' = M/2$, $V' = (-\infty,N]\times[-M',M'] \cap \mbb{Z}^2$, and $\zeta' = \min\{n \geq 1\,:\, S_n \in \partial V'\}$. For an integer $k$, let $v_k = 1 + ik$. By the invariance principle, there exists a constant $c_a>0$ such that
\[
\mbb{P}_{v_k}(\Re(S_{\zeta'}) = N) \geq c_a, \quad |k| \leq M'/2.
\]

Let $D = \{v_k\,:\, |k| \leq M' -1 \}$, $\tau = \min\{n \geq 1 \,:\, S_n \in D\} \wedge \zeta'$, and $G$ be the Green's function on $V'$. By a last exit decomposition, we have for $v \in D$
\[
\mbb{P}_{v}(\Re(S_{\zeta'}) = N) = \sum_{v' \in D} G(v,v') \mbb{P}_{v'}(\Re(S_\tau) = N).
\]
By \cite[Lemma 1]{DingLi18}, $\max\{G(v, D)\,:\, v \in D\} \leq c_a' N$. Therefore, combining the last two displays and summing over $v \in D$ gives
\[
c_a N \leq \sum_{v \in D} \mbb{P}_{v}(\Re(S_{\zeta'}) = N) \leq c_a' N\sum_{v \in D} \mbb{P}_{v}(\Re(S_{\tau}) = N).
\]
Finally, we note that for any $v \in D$, $\mbb{P}_v(\Re(S_\tau) = N) \leq \mbb{P}_1(\Re(S_\zeta) = N)$, so the conclusion follows.
\end{proof}

%%%%%%%%%%%%%%%%%%%%%%%%%%%%%
\section{Estimates of crossing probabilities in the metric graph}
\label{MGFF RSW}

\subsection{The zero-boundary case}
\label{subsec::MGFF_zerobc}
In this section we prove \eqref{MGFF zero boundary result}. The proof consists of analyzing the exploration martingale $M$, introduced in Section~\ref{subsec::pre_explorationmart}, corresponding to an exploration on $\til E^{\geq 0}_\delta$ with $\mc I_0 = [c'_\delta,d'_\delta)$ and $U = \{u \in \partial V_{\delta}^o\,:\, \Re(u) \leq L/4\}$. 
%Notice that $(a'_\delta,b'_\delta) \subseteq U$ but for technical reasons it is convenient to take a larger subset of $\partial V^o$. 
Recall that
\[M_t = \E[X_U \mid \mc F_{\mc I_t}]=\sum_{v\in\mc{I}_t}\Hm_t(U, v)\tilde{\phi}_{\delta}(v).\]
We let $\mc I_0^- = \{v \in \mc{I}_0\,:\, \til\phi_\delta(v) < 0 \}$. Since $\til\phi_\delta$ is non-negative on $\mc I_t \setminus \mc I_0^-$, we have 
\[
M_t \geq \sum_{v \in \mc I_0^-} \Hm_t(U,v) \til\phi_\delta(v). 
\]
Noting that $\Hm_t(U,v)$ is decreasing in $t$ for $v \in \mc I_0^-$ (since $\mc I_t$ is increasing), we conclude
\begin{equation}\label{eqn::mgffzero_mart_lowerbound}
M_t - M_0 \geq -\sum_{v \in \mc I_0} \Hm_0(U,v) (\til\phi_\delta(v)\vee 0).
\end{equation}

Next, we turn to bounding the quadratic variation. In particular, we claim that there exists $c = c(L) > 0$ such that
\begin{equation}\label{Horizontal crossing QV lower bound}
\left\{[a_\delta',b_\delta') \stackrel{\til\phi_\delta \geq 0}{\longleftrightarrow}[c'_\delta,d'_\delta)\right\} \subseteq \left\{\langle M \rangle_\infty \geq c |\log(\delta)|\right\}.
\end{equation}

Before proving~\eqref{Horizontal crossing QV lower bound}, we show how it can be used to conclude the proof of~\eqref{MGFF zero boundary result}. Suppose $(B_s)_{s\ge 0}$ is a standard one-dimensional Brownian motion. 
Combining Theorem~\ref{Time-change theorem} with~\eqref{eqn::mgffzero_mart_lowerbound} and~\eqref{Horizontal crossing QV lower bound},  we have
\begin{align*}
\mbb{P}\left([a_\delta',b_\delta') \stackrel{ \til\phi_{\delta}\ge 0}{\longleftrightarrow}[c'_\delta,d'_\delta)\right) &\leq \mbb{P}\left(\inf_{0 \leq s \leq c|\log(\delta)|} B_s \geq - \sum_{v \in \mc I_0} \Hm_0(U,v) (\til\phi_\delta(v)\vee 0)\right)\\
&\leq \frac{1}{\sqrt{c|\log(\delta)|}} \sum_{v \in \mc I_0} \Hm_0(U,v)\E[\til\phi_\delta(v)\vee 0]\\
&\leq \frac{c'}{\sqrt{|\log(\delta)|}}\Hm(U,\mc I_0),
\end{align*}
where in the second inequality we used the fact that, for $t$ and $a>0$, we have by Proposition~\ref{prop::Hitting times with drift} that $\mbb{P}\left(\inf_{0\le s\le t}B_s\ge -a\right)\le a/\sqrt{t}$ ; and in the third inequality we used the fact that $G(v,v) \leq 4$ for $v \in \mc I_0$ (since $v$ is adjacent to $\partial V_\delta$). It remains to bound $\Hm(U,\mc I_0)$. Since $U, \mc I_0 \subset V_\delta$, we can consider a discrete time simple random walk $S$ (on $\delta \mbb{Z}^2$), killed on $\partial V_\delta$ (instead of the metric graph Brownian motion $\til B$). For a vertex $u \in U$, we have $\Hm(u,\mc I_0)\le c\delta$ by Lemma~\ref{lm-Box Hm upper bound}. Since $|U| \leq c/\delta$, we conclude
\[
\Hm(U, \mc I_0) \leq c.
\]
Thus,the proof will be complete once we prove \eqref{Horizontal crossing QV lower bound}.

\begin{proof}[Proof of \eqref{Horizontal crossing QV lower bound}]
We begin by noting that on $\left\{[a'_\delta,b'_\delta) \stackrel{\til\phi_\delta \geq 0}{\longleftrightarrow} [c'_\delta,d'_\delta)\right\}$ there exists a vertex $u^* \in [a'_\delta,b'_\delta)$ such that the set $\mc{I}_\infty$ contains a nearest neighbor path $\gamma \subset V_\delta^o$ connecting $u^*$ to the complement of the box of radius $L/4$ around $u^*$. We will show that for any such $u^*$ and any such $\gamma$, we have
\[
\var[X_U \mid \mc{F}_{\mc I_0}] - \var[X_U \mid \mc{F}_{\gamma \cup \mc I_0}] \geq c|\log(\delta)|.
\]
Recall that $G_0$ denotes the Green's function on $\tilde{V}_{\delta}\setminus \mc I_0$ and that
\[
\var[X_U \mid \mc{F}_{\mc I_0}] - \var[X_U \mid \mc{F}_{\gamma \cup \mc I_0}] = \sum_{v \in \gamma}\Hm_0(U,v; \gamma) G_0(v,U).
\]
Let $S$ be a discrete-time simple random walk (on $\delta \mbb{Z}^2$) killed on $\partial V_\delta \cup \mc I_0$ and $T_U$ be the hitting time of $U$ by $S$. It is easy to see that
\[
G_0(v,U) \geq \mbb{P}_v(T_U < \infty) \geq c, \quad \forall v \in \gamma,
\]
for some universal constant $c > 0$. Therefore, it remains to show
\begin{equation}\label{eq-path to boundary Harmonic measure}
\Hm_0(U,\gamma) \geq c |\log(\delta)|.
\end{equation}
To this end, we partition $U$ and $\gamma$ as follows. For $n\geq0$ an integer, we let $Q_n$ be the box of radius $r_n = 2^{-n-2}L$ centered at $u^*$, $A_n = Q_n \setminus Q_{n+1}$, $U_n = U \cap A_n$, and $\gamma_n = \gamma \cap A_n$. We note that $N = \max\{n \geq 1\,:\, r_n \geq 1000\delta\}$ satisfies $N \geq c |\log(\delta)|$. Finally, we claim there exists a universal constant $c > 0$ such that
\[
\Hm_0(U_n,\gamma_n; \gamma) \geq c,\quad 0 \leq n \leq N\,.
\]
For the proof, let $u_1 \in U_n$ be such that $2r_n/3 - \delta < |u_1 - u^*|_{\ell_\infty} \leq 2r_n/3$, $u_2 \in U_n$ be such that $5r_n/6 \leq |u_2 - u^*|_{\ell_\infty} < 5r_n/6 + \delta$, and $Q_{n,2}$ be a box of radius $r_n/12$ centered at $u_2$. With these choices, the distances between $u_1$ and $Q_{n,2}$, and between $\{u_1\} \cup Q_{n,2}$ and $A_n^c$ are of the same order as $r_n$. Therefore, if we let $S$ be a simple random walk and $\mathfrak E$ be the event that $S$ hits $Q_{n,2} \cap U$ and then completes a loop around $u^*$ before exiting $A_n$, we have $\mbb{P}_{u_1}(\mathfrak E) \geq c$. Note that such a walk necessarily contains a path from $U_n$ to $\gamma_n$ which does not hit $\partial V_\delta$. Therefore, letting $T$ be the hitting time of $U_n \cup \gamma_n$ by $S$, and $\zeta$ be the hitting time of $A_n^c$, we have by a last exit decomposition
\[
\mbb{P}_{u_1}(\mathfrak E) \leq \sum_{u \in U} G^*_n(u_1,u)\mbb{P}_u(S_{T \wedge \zeta} \in \gamma_n),
\]
where $G^*_n(u_1,u)$ is the expected number of visits a random walk started at $u_1$ makes to $u$ after hitting $Q_{n,2}\cap U_n$, and before exiting $A_n$. It is immediate that $\mbb{P}_u(S_{T \wedge \zeta} \in \gamma_n) \leq \Hm(u,\gamma_n ; \gamma)$, and we have that
\[
G_n^*(u_1,u) \leq G_n(u_1,u) \wedge \max\{G_n(u',u) \,:\, u' \in Q_{n,2} \cap U_n\},
\]
where $G_n$ is the Green's function on $A_n \cap \delta \mbb{Z}^2$. It follows from \cite[Theorem 4.4.4, Proposition 4.6.2]{LawlerLimic10} that $G^*_n$ is uniformly bounded. That is, there exists a universal constant $c$ such that
\[
G_n^*(u_1,u) \leq c,\quad u \in U_n.
\]
Finally, this implies that
\[
\Hm_0(U_n,\gamma_n; \gamma) \geq \sum_{u \in U_n} \mbb{P}_{u}(S_{T \wedge \zeta} \in \gamma_n)\geq c \mbb{P}_{u_1}(\mathfrak E) \geq c.
\]
This concludes the proof.
\end{proof}

\subsection{The alternating boundary case}\label{subsec::alternating MGFF proof}
\label{subsec::MGFF_alternatingbc}
In this section we prove the lower bound in~\eqref{MGFF alternating boundary result}. The proof consists of two main claims, both of which are proved using an exploration martingale. 
Recall that $\til{\A}_{\delta,0}$ (resp. $\til{\AB}_{\delta,0}$) is the first passage set of $\tilde{\phi}_{\delta}$ above zero (resp. below zero). First, letting $\Pi = \{z \,:\, |\Im(z)-1/2|\le 1/4\}$, we claim that there exists a constant $c_1=c_1(L, \lambda)>0$ such that
\begin{equation}\label{Control on negative cluster}
\mbb{P}\left(\til{\AB}_{\delta,0} \cap \Pi = \emptyset\right) \geq c_1.
\end{equation}                                                                                                                                                                                                           
Next, we claim that conditional on this event, the probability that there is a positive crossing is bounded uniformly away from zero. That is, letting $\mbb{P}^+$ denote the law of $\til{\phi}_{\delta}$ given $\mc{F}_{\til{\AB}_{\delta,0}}$ (and $\E^+$ be the expectation with respect to $\mbb{P}^+$), there exists a constant $c_2=c_2(L, \lambda)>0$ such that
\begin{equation}\label{Plus-zero RSW}
\mbb{P}^+\left([a_\delta,b_\delta) \stackrel{\til{\A}_{\delta,0}}{\longleftrightarrow} [c_\delta, d_\delta)\right) \geq c_2 \quad \text{a.s. on } \{\til{\AB}_{\delta,0} \cap \Pi = \emptyset\}.
\end{equation}

Assuming~\eqref{Control on negative cluster} and~\eqref{Plus-zero RSW}, we obtain the lower bound in~\eqref{MGFF alternating boundary result}: 
\[
\mbb{P}\left([a_\delta,b_\delta) \stackrel{\til{\A}_{\delta,0}}{\longleftrightarrow} [c_\delta, d_\delta)\right) \ge \mbb{P}\left([a_\delta,b_\delta) \stackrel{\til{\A}_{\delta,0}}{\longleftrightarrow} [c_\delta, d_\delta), \til{\AB}_{\delta,0} \cap \Pi = \emptyset\right)\ge c_1c_2.
\]
Thus it remains to show~\eqref{Control on negative cluster} and~\eqref{Plus-zero RSW}.

\begin{proof}[Proof of~\eqref{Control on negative cluster}]
For this proof, we consider an exploration martingale $M$, introduced in Section~\ref{subsec::pre_explorationmart}, corresponding to an exploration of $\til E^{\leq 0}_\delta$ from $\mc I_0 = [b_\delta,c_\delta) \cup [d_\delta,a_\delta)$ with observable corresponding to 
\[
U = \left\{u \,:\, u \in [a'_\delta,b'_\delta) \cup [c'_\delta,d'_\delta), \Big|\Im(u) - \frac{1}{2} \Big| \leq \frac{3}{8}\right\}
\]
The following processes will be useful in the analysis
\[
\pi_t = \Hm_t(U,[a_\delta,b_\delta) \cup [c_\delta,d_\delta)), \quad \mu_t = \Hm_t(U,\mc I_t).
\]
Note that we have $M_0 = \lambda(\pi_0 - \mu_0)$ and $M_t \leq \lambda \pi_t$. Since $\pi_t$ is decreasing,
\[
M_t - M_0 \leq \lambda \mu_0.
\]
Note that by Lemma~\ref{lm-Box Hm upper bound} there exists $c > 0$ such that $\Hm(u,\mc I_0) \leq c\delta$ for all $u \in U$, which implies $\mu_0 \leq c$. Let $T = \til{D}^{\leq 0}_\delta(\mc I_0,\Pi)$ be the time that the exploration reaches $\Pi$. On $\{\til{\AB}_{\delta,0} \cap \Pi \neq \emptyset\}$, the set $\mc I_T$ contains a nearest neighbor path $\gamma \subset V_\delta^0$ crossing one of the strips making up the region
\[
\Pi' = \left\{z \,:\, \frac{1}{4}\leq \Big|\Im(z) - \frac{1}{2}\Big| \leq \frac{3}{8}\right\}.
\]
We claim that for there exists a constant $c > 0$ such that for any such path
\begin{equation}\label{eq-Alternating MGFF QV lower bound 1}
\var(X_U) - \var(X_U \mid \mc{F}_{\gamma}) \geq c.
\end{equation}
Assuming this bound for now, we have
\[
\langle M\rangle_T \geq \min_{\gamma} \{\var(X_U) - \var(X_U \mid \mc{F}_{\gamma}) \}\geq c,
\]
almost surely on $\{\til{\AB}_{\delta,0} \cap \Pi \neq \emptyset\}$ (here we used the fact that $\mc{F}_{\mc I_0}$ is the trivial $\sigma$-field since $\mc I_0 \subset \partial V_\delta$). Applying Theorem~\ref{Time-change theorem} and Proposition~\ref{prop::Hitting times with drift} gives
\[
\mbb{P}(\til{\AB}_{\delta,0} \cap \Pi \neq \emptyset) \leq \mbb{P}\left(\sup_{0 \leq t \leq c} B_t \leq c'\lambda\right) \leq 1 - c''.
\]
Turning to the proof of \eqref{eq-Alternating MGFF QV lower bound 1}, we have by the invariance principle that $G(v, U) \geq c$ for any $v \in \gamma$. Therefore
\[
\var(X_U) - \var(X_U \mid \mc{F}_{\gamma}) = \sum_{v \in \gamma} \Hm(U,v; \gamma) G(v, U) \geq c \Hm(U,\gamma).
\]
Finally, by Lemma~\ref{lm-Box Hm lower bound} there exists a constant $c > 0$ such that $\Hm(u,\gamma) \geq c\delta$ for all $u \in U \cap \Pi$. Since $|U \cap \Pi| \geq c \delta^{-1}$, it follows that $\Hm(U, \gamma) \geq c$. This concludes the proof.
\end{proof}

\begin{proof}[Proof of \eqref{Plus-zero RSW}]
Let $\til{V}_\delta^+$ be the connected component of $\til{V}_\delta \setminus \til{\AB}_{\delta,0}$ containing $[a_\delta,b_\delta) \cup [c_\delta,d_\delta)$, $V_\delta^+ = \til{V}_\delta^+ \cap \delta\mbb{Z}^2$ be the vertices in $\til{V}_\delta^+$, and $U = [c'_\delta,d'_\delta) \cap V_\delta^+$. For any $u \in U$, let $\adj(u)$ be the (unique) vertex in $[c_\delta,d_\delta)$ adjacent to $u$. For $0 < \epsilon \leq 1$, let $u_\epsilon = \epsilon u + (1-\epsilon)\adj(u)$ and
\[
U_\epsilon = \{ u_\epsilon\,:\, u\in U\}.
\]
Finally, we let $M_\epsilon$ be the exploration martingale corresponding to an exploration on $\til{E}^{\geq 0}_\delta$ from $\cI_0 = [a_\delta,b_\delta)$ with obsevable $X_{U_\epsilon}$. As in the proof of \eqref{Control on negative cluster}, we will use the processes
\[
\pi_{\epsilon,t} = \Hm^+_t(U_\epsilon, [c_\delta,d_\delta)), \quad \mu_{\epsilon,t} = \Hm^+_t(U_\epsilon, \mc{I}_t),
\]
where $\Hm_t^+$ is the harmonic measure on $\mc{I}_t \cup \partial \til V_\delta^+$. That is, letting $\til\zeta^+$ be the hitting time of $\partial \til V_\delta^+$ and $\tau_t$ be the hitting time of $\mc{I}_t$, $\Hm_t^+(u,v) = \mbb{P}_u(\til B_{\tau_t \wedge \til\zeta^+}= v)$. Note that $Hm^+_0$ is simply the harmonic measure on $\partial \til V_\delta^+$ so we will write $\Hm^+$ in this case instead. Note that $M_{\epsilon,t} \geq \lambda \pi_{\epsilon,t}$ and $M_{\epsilon,0} = \lambda (\mu_{\epsilon,0} + \pi_{\epsilon,0})$, so we obtain
\[
M_{\epsilon,t} - M_{\epsilon,0} \geq -\lambda(\mu_{\epsilon,0} + \pi_{\epsilon,0} - \pi_{\epsilon,t}),
\]
with equality if and only if $M_{\epsilon,t} = M_{\epsilon, \infty}$ and $\mc I_t \cap U_\epsilon = \emptyset$ (i.e. the exploration has stopped by time $t$ before hitting $U_\epsilon$). To conclude the proof, we need to lower bound $\mu_{\epsilon,0}$ and upper bound $\langle M_\epsilon \rangle_t$.

First, we claim that there exists $c_1 > 0$ such that $\mu_{\epsilon,0} \geq c_1 \epsilon$. Indeed, we have by Lemma~\ref{lm-Box Hm lower bound} and the assumption $\til{\AB}_{\delta,0} \cap \Pi = \emptyset$ that for any $u \in U$ such that $|\Im(u) - 1/2| \leq 1/8$
\[
\Hm^+(u, [a_\delta,b_\delta)) \geq c_1 \delta.
\]
It follows that $\Hm^+(U, [a_\delta,b_\delta)) \geq c_1$. Next, by construction a metric graph Brownian motion started at $u_\epsilon$ will hit $u$ before $\adj(u)$ with probability $\epsilon$. This gives 
\[
\mu_{\epsilon,0} = \epsilon\Hm^+(U,[a_\delta,b_\delta)) \geq c_1\epsilon.
\]

Second, for the upper bound on the quadratic variation, we claim that there exists $c_2 > 0$ such that 
\[
\langle M_\epsilon \rangle_t \leq c_2 \epsilon (\pi_{\epsilon,0} - \pi_{\epsilon,t}), \quad \forall t \leq \til{D}^{\geq 0}_\delta([a_\delta,b_\delta),U_\epsilon).
\]
For the proof, note that for such $t$ we have 
\begin{align*}
\pi_{\epsilon,0} - \pi_{\epsilon,t} &= \sum_{v \in \partial \til{\mc{I}}_t} \Hm^+_t(U_\epsilon,v)\Hm^+(v,[c_\delta,d_\delta)), \\
\langle M_\epsilon \rangle_t &= \sum_{v \in \partial \til{\mc{I}}_t} \Hm^+_t(U_\epsilon,v)G^+(v,U_\epsilon),
\end{align*}
where $G^+$ is the Green's function of a metric graph Brownian motion killed on $\partial\til V_\delta^+$. To proceed, let $\tau_{u_\epsilon}$ be the hitting time of $u_\epsilon$ and recall $\til\zeta^+$ is the hitting time of $\partial \til V_\delta^+$. We have
\begin{align*}
G^+(v,U_\epsilon) &= \sum_{u \in U}\mbb{P}_v(\tau_{u_\epsilon} < \til\zeta^+) G^+(u_\epsilon,u_\epsilon),\\
	&= \epsilon \sum_{u \in U} \mbb{P}_v(\tau_{u_\epsilon}<\til\zeta^+)[ \epsilon G^+(u,u) + 4(1-\epsilon)],\\
	&\leq 4\epsilon\sum_{u \in U}\mbb{P}_v(\tau_{u_\epsilon}<\til\zeta^+), 
\end{align*}
where in the second equality we used \eqref{Metric Green Interpolation} to express $G^+(u_\epsilon, u_\epsilon)$ as a sum of terms involving $u$ and $\adj(u)$. Since $\adj(u) \in \partial \til V_\delta^+$ we have $G^+(u,\adj(u)) = G^+(\adj(u),\adj(u)) = 0$. In the inequality we used the fact that $G^+(u,u) \leq 4$ since $u$ is adjacent to $\partial \til V_\delta^+$. Finally, assuming from now on that $\epsilon < 1/2$, we have
\[
\Hm^+(v,\adj(u)) \geq \frac{1}{2}\mbb{P}_v(\tau_{u_\epsilon} < \til\zeta^+).
\]
It follows that $\langle M_\epsilon \rangle_t \leq 16\epsilon(\pi_{\epsilon,0} - \pi_{\epsilon,t})$. Putting both bounds together we conclude that (conditional on $\{\til{\AB}_{\delta,0} \cap \Pi = \emptyset\}$)
\[
\left\{M_{\epsilon,t} - M_{\epsilon,0} \geq  - \frac{c_2\lambda}{\epsilon}\langle M_\epsilon\rangle_t - c_1\lambda\epsilon,\, 0 \leq t \leq \til{D}^{\geq 0}_\delta([a_\delta,b_\delta),U_\epsilon) \right\} \subseteq \{[a_\delta,b_\delta) \stackrel{\til{\A}_{\delta,0}}{\longleftrightarrow}U_\epsilon\}.
\]
Applying Theorem~\ref{Time-change theorem} and Proposition~\ref{prop::Hitting times with drift} (after re-scaling the Brownian motion) we obtain for some $c > 0$ (independent of $\epsilon$),
\[
\mbb{P}^+\left([a_\delta,b_\delta) \stackrel{\til \phi_\delta \geq 0}{\longleftrightarrow}U_\epsilon\right) \geq
	\mbb P(B_t \geq -\lambda (c_2 t + c_1), t \geq 0) \geq c,
\]
where $B$ is a standard Brownian motion. Letting $\epsilon \to 0$ concludes the proof.
\end{proof}

%%%%%%%%%%%%%%%%%%%%%%%%%%%%%%%
\section{Estimates of crossing probabilities in the discrete graph}
\label{DGFF RSW}
The goal of this section is to prove \eqref{DGFF zero boundary result} and \eqref{Alternating boundary difference result}. Both proofs begin with the observation that in the discrete graph, when we explore the excursion set $E_\delta^{>0}$ starting from $\mathcal I_0$, the set $\partial \mathcal I_\infty \setminus \partial V_\delta$ is contained in $E^{<0}_\delta$. We show that in fact, suitably defined averages of the field on this set are bounded away from 0 with high probability. Following the literature on the subject, we call this phenomenon entropic repulsion. It was previously studied in a similar context in \cite{DingLi18}.

\subsection{Entropic repulsion for explorations on the discrete graph}
In this section we formalize the idea, stated above, that when there is no horizontal crossing the values of the field on the outer boundary of the explored set are strictly negative. We begin with some setup. For this section, we let $\til{K} \subset \til{ \mbb{Z}}^2$ be compact and simply connected and $V = \til{K} \cap \mbb{Z}^2$. For $U, I \subset V$, we let $D$ and $D'$ be
\[
D = \{v \in I \,:\, \Hm(U,v;I) > 0\}, \quad D' = \{v \in I \,:\, \Hm(U,v;I \setminus D) > 0\},
\]
where $\Hm(\cdot,\cdot; W)$ is the harmonic measure on $W \cup \partial \til K$. That is, $D$ is the subset of $\partial I$ that can be reached from $U$, and $D'$ is the subset of $\partial (I \setminus D)$ that can be hit from $U$ (equivalently from $D$). As the notation suggests, $U$ corresponds to the ``observable'' in the definition of the exploration martingale, while $I$ corresponds to the explored set. We will assume the pair $(U,I)$ satisfies some geometric conditions. First, that for any $v \in D$, all edges incident on $v$ are contained in $\til{K}$. That is, $\dist(v, \partial \til{\mc{K}}) \geq 1$, where $\dist$ denotes the Euclidean distance. Second, that for any $v \in D$, there exists $w \in I \setminus D$ such that $v \sim w$. We let $\xi$ be given by
\[
\xi(U, I) = \frac{1}{\Hm(U,I)} \sup\{\Hm(U,w;I \setminus D) \,:\, w \in D'\}.
\]
The reason the supremum is over $w \in D'$ is the following. For any $v \in D$ there exists $w \in D'$ such that $v \sim w$ which gives
\[
\Hm(U,w; I\setminus D) \geq \frac{1}{4}\Hm(U,v;I).
\]
Therefore, $\Hm(U,v;I) \leq 4 \xi(U,I) \Hm(U,I)$ for any $v \in D$. Finally, we let $\til\phi$ be a metric graph GFF on $\til{K}$ with boundary condition $f$ satisfying $|f(v)| \leq \lambda$ for all $v \in \partial \til{K}$. As usual, we let $\phi$ be the restriction of $\til\phi$ to $V$. Note that $\phi$ is a discrete GFF on $V \cup \partial \til{K}$ with boundary condition $f$. The goal is to show that when $\xi(U,I)$ is small and we condition $\phi$ to be negative on $D$, the typical value of the field on $D$ (as seen from $U$) is bounded away from zero with high probability. 

\begin{proposition}\label{Entropic repulsion result}
Let $\til{K}$, $U$, and $I$ satisfy the conditions above. Additionally, let $I^+$ and $I^-$ be a partition of $I$ such that $D \subseteq I^-$ and for every $v \in D$, there exists $w \in I^+$ such that $v \sim w$. Let $\mathfrak E$ be the following event
\[
\mathfrak{E} = \{\til\phi(v) < 0 \,:\, v \in I^-\} \cap \{\til\phi(v) \geq 0\,:\, v \in I^+\}.
\]
Let $Y$ be the following random variable
\[
Y = \sum_{v \in D} \Hm(U,v;I)\phi(v).
\]
There exists a universal function $r$ satisfying $r(x) \to 0$ as $x \to 0$ and a constant $\Delta = \Delta(\lambda) > 0$ such that
\[
\mbb{P}(Y \leq -\Delta \Hm(U,I) \mid \mathfrak E) \geq 1 - r(\xi(U,I)).
\]
\end{proposition}
The proof follows along the same lines as \cite[Lemma 6]{DingLi18}, and consists of bounding the conditional mean and variance of $Y$ given $\mathfrak E$ and the field on $I^+$. We isolate three lemmas that will be used repeatedly throughout the rest of the paper.
\begin{lemma}\label{lm-Monotonicity to conditioning}
Let $\{B_{1,v}\}_{v \in V}$ and $\{B_{2,v}\}_{v \in V}$ be a sequence of non-empty intervals satisfying
\begin{align*}
\inf B_{1,v} \leq \inf B_{2,v}, \quad \sup B_{1,v} \leq \sup B_{2,v}, \quad \forall v \in V.
\end{align*}
Let $\til\phi_1$ and $\til\phi_2$ be metric graph GFFs on $\til{K}$ with boundary conditions $f_1$ and $f_2$ such that $f_1(v) \leq f_2(v)$ for $v \in \partial \til{K}$, and let $\phi_1$ and $\phi_2$ be the corresponding discrete graph GFFs. We assume that $B_{j,v} = \{f_j(v)\}$ for $v \in \partial \til{K} \cap V$. Finally, for $j = 1,2$ let $\mathfrak E_j$ be the event
\[
\mathfrak E_j = \{\phi_j(v) \in B_{j,v}, v \in V\}.
\]
Then the law of $\phi_1$ given $\mathfrak E_1$ is stochastically smaller than the law of $\phi_2$ given $\mathfrak E_2$.
\end{lemma}

\begin{proof}%[Proof of Lemma~\ref{lm-Monotonicity to conditioning}]
The proof is the same as the proof of \cite[Equation (49)]{DingLi18}; we reproduce it here for completeness. Let $\nu_1$ be the law of $\phi_1$ given $\mathfrak E_1$ and $\nu_2$ be the law of $\phi_2$ given $\mathfrak E_2$. Further, let $\til{K}'$ be the closure of an open neighborhood of $\til{K}$ such that $\til{K}' \cap \mbb{Z}^2 = V$. Let $\til\phi'$ be a metric graph GFF on $\til{K}'$ with zero boundary condition, say, and $\phi'$ be the corresponding discrete GFF on the graph with vertex set $V' = V \cup \partial \til{K}$. We note that $\mu$, the law of $\phi'$, has density $\mu(dr) = \exp(-H(r))dr$ (where $r$ is a $|V'|$-dimensional vector) such that for every $r,r' \in \mbb{R}^{|V'|}$
\[
H(r \wedge r') + H(r \vee r') \leq H(r) + H(r'),
\]
where $\wedge$ and $\vee$ are taken coordinate by coordinate. Additionally, we note $\nu_j$ is simply the law of the restriction of $\phi'$ to $V$, conditioned on the event
\[
\mathfrak E_j' = \{\phi'(v) = f_j(v), v \in \partial \til{K}\} \cap \{\phi'(v) \in B_{j,v}, v \in V\}.
\]
We let $\mu_j$ be the law of $\phi'$ conditioned on $\mathfrak E'_j$. For $q > 0$ and $B \subset \mbb{R}$, we define the function
\begin{equation}\label{eq-Wq}
W^{(q)}_B(t) = q\dist(t,B)^4.
\end{equation}
We can then approximate $\mu_1$ and $\mu_2$ by probability measures satisfying the following
\begin{align*}
\mu^{(q)}_j(dr) &\propto \exp\left(-\sum_{v \in V \cup \partial \til{K}} W^{(q)}_{B_{j,v}}(r_v)\right)\mu(dr),
\end{align*}
where we set $B_{j,v} = [f_j(v)]$ for $v \in \partial \til{K}$. It is clear that for any $t,t' \in \mbb{R}$ and any $v \in V \cup \partial\til{K}$ we have
\[
W^{(q)}_{B_{1,v}}(t \wedge t') + W^{(q)}_{B_{2,v}}(t \vee t') \leq W^{(q)}_{B_{1,v}}(t) + W^{(q)}_{B_{2,v}}(t').
\]
Therefore, it follows from \cite{Preston74} that $\mu_2^{(q)}$ is stochastically smaller than $\mu_1^{(q)}$ for any $q > 0$. As $q \to \infty$, $\mu_i^{(q)}$ converges weakly to $\mu_i$, so it follows $\mu_2$ is stochastically smaller than $\mu_1$. Finally, this implies that $\nu_1$ is stochastically smaller than $\nu_2$.
\end{proof}
In light of this lemma (and the fact that $Y$ is an increasing function of $\phi$) we will assume, without loss of generality, that $I^- = D$ and $I^+$ is the set of all vertices in $V$ separated from $U$ by $D$. That is,
\[
I^+ = \{v \in V \,:\, \Hm(U,v;D \cup \{v\}) = 0\}.
\]
With this assumption, we are ready to begin the moment analysis. We start with the variance as the argument is easier.
\begin{lemma}\label{lm-Conditional variance}
\[
\var[Y \mid \mc{F}_{I^+}, \mathfrak E] \leq \var[Y \mid \mc F_{I^+}] \quad \text{a.s.}
\]
\end{lemma}
\begin{proof}%[Proof of Lemma~\ref{lm-Conditional variance}]
The proof is the same as that of \cite[Equation (56)]{DingLi18}; we reproduce it here for completeness. Let $\mathbf Y = (\phi(v))_{v \in D}$, $\mu$ be the law of $\mathbf Y$ given $\mc F_{I^+}$, and $\nu$ be the law of $\mathbf Y$ given $\mc F_{I^+}$ and $\mathfrak E$. Let $m$ and $\Sigma$ be the mean and variance of $\mathbf Y$ given $\mc F_{I^+}$ and note that $\Sigma$ is deterministic. Recall $\mu(dr) \propto \exp(-\frac{1}{2}(r - m) \Sigma^{-1} (r-m))dr$. For $q > 0$, we approximate $\nu$ by a probability measure $\nu^{(q)}$ satisfying the following
\[
\nu^{(q)} \propto \exp\left(-q\sum_{v \in D} r_v^4 \one_{r_v \geq 0}\right) \mu(dr).
\]
Since the second derivative of $f(t) = t^4 \one_{t \geq 0}$ is non-negative, we have that $\nu^{(q)}$ is of the form $\nu^{(q)}(dr) = \exp(-H(r))dr$ where $\inf_r \text{Hess}(H)(r) \geq \frac{1}{2}\Sigma^{-1}$. Therefore, by the Brascamp-Lieb inequality \cite{BrascampLieb76}, for a random vector $\mathbf Y^{(q)} \sim \nu^{(q)}$ and any $l \in \mathbb{R}^{|D|}$, we have $\var[l \cdot \mathbf Y^{(q)}] \leq \var[l \cdot \mathbf Y \mid \mc F_{I^+}]$. As $q \to \infty$, $\nu^{(q)}$ converges weakly to $\nu$, so we have 
\[
\var[l \cdot \mathbf Y \mid \mc F_{I^+}, \mathfrak E] \leq \var[l \cdot \mathbf Y \mid \mc F_{I^+}] \quad \text{a.s.}
\]
Since $Y$ is of the form $l \cdot \mathbf Y$, the conclusion follows.
\end{proof}
We obtain from this
\begin{corollary}\label{cor-Conditional variance}
\[
\var[Y \mid \mc{F}_{I^+}, \mathfrak E] \leq 16 \xi(U,I) \Hm(U,I)^2, \quad \text{\emph{a.s.}}
\]
\end{corollary}
\begin{proof}
We note
\begin{align*}
\var[Y \mid \mc{F}_{I^+}] &= \sum_{v,v' \in D}\Hm(U,v;I)\Hm(U,v';I) G_{\til{K} \setminus I^+}(v',v),\\
	&=\sum_{v \in D} \Hm(U,v;I)G_{\til{K} \setminus I^+}(U,v)
\end{align*}
where $G_{\til{K} \setminus I^+}$ is the Green's function on $\til{K} \setminus I^+$ and we have used the fact that
\[
\sum_{v' \in D} \Hm(U,v';I)G_{\til{K} \setminus I^+}(v',v) = G_{\til{K} \setminus I^+}(U,v), \,\, \forall v \in D.
\]
Note that 
\[
G_{\til{K} \setminus I^+}(U,D) \leq 4 \Hm(U,I),
\]
where we have used the fact that a random walk started on $D$ will hit $I^+$ with probability at least $1/4$ before returning to $D$. Therefore, we have
\[
\var[Y \mid \mc{F}_{I^+}] \leq 4 \Hm(U,I) \sup_{v \in D} \Hm(U,v;I) \leq 16 \xi(U,I) \Hm(U,I)^2.
\]
\end{proof}
To conclude, we need a high-probability bound on the conditional expectation of $Y$ given $\mc{F}_{I^+}$ and $\mathfrak E$. We begin with an auxiliary lemma.
\begin{lemma}\label{lm-Conditional expectation}
There exists a universal continuous, increasing function $g:[0,\infty) \to (0,\infty)$ such that for any $w \in \til{K}^o$ and any $\epsilon > 0$ the following holds. Let $\mathfrak E_w$ be the event
\[
\mathfrak E_w = \{\til\phi(w) = 0\}\cap \{\til\phi(u) > 0\, \forall u \in \til{K},\, \dist(u, \partial\til{K} \cup \{w\}) \geq \epsilon\}.
\]
For $v \in \til{K}$ such that $|v - w| \leq r$, we have
\[
\E[\phi(v) \mid \mathfrak E_w] \leq g(r) + \lambda.
\]
\end{lemma}
\begin{proof}
We follow the proof of \cite[Lemma 6]{DingLi18}. First, by an argument similar to the proof of \cite[Equation (48)]{DingLi18}, there exists a function $R : \mathbb{R}^+ \to \mathbb{R}^+$ such that for any $a,b > 0$, there exists a function $h_{a,b}$ that is harmonic on $\til{\mbb{Z}}^2 \setminus \{w\}$ and satisfies
\[
\Big|h_{a,b}(u) - a\log\left(|u-w| + 2\right) - b\Big| \leq R(a), \quad \forall u \in \til{\mbb{Z}}^2.
\]
Let $\til\varphi_{a,b}$ be a Gaussian free field on $\til{K}$ with boundary condition $h_{a,b}$ on $\partial \til{K} \cup \{w\}$, and $\varphi_{a,b}$ be the restriction of $\til\varphi$ to $V$. Let $A \subset \til K$ be a countable, dense subset of $\{u \,:\, u \in \til K, \, \dist(u, \partial \til K \cup \{w\}) \geq \epsilon\}$, and $\{A_n\}_{n = 1}^\infty$ be an increasing sequence of finite subsets of $A$ such that $\cup_{n = 1}^\infty A_n = A$. Applying Lemma~\ref{lm-Monotonicity to conditioning} to $A_n$ and taking a limit, we have that for any $a,b > 0$ such that $h_{a,b}(u) > 0$ for all $u \in \til{\mbb{Z}}^2$, the following holds
\[
\E[\til\phi(v) \mid \mathfrak E_w] \leq \E[\til\varphi_{a,b+\lambda}(v) \mid \til\varphi_{a,b+\lambda}(u) > 0, \forall u \in \til {K},\, \dist(u, \partial\til{K} \cup \{w\}) \geq \epsilon].
\]
By the FKG inequality,
\[
\E[\til\phi(v) \mid \mathfrak E_w] \leq \E[\til\varphi_{a,b+\lambda}(v) \mid \til\varphi_{a,b+\lambda}(u) > 0, \forall u \in \til K].
\]
Next, we claim there exist universal constants $c_a,c_b >0$ such that for all $a\geq c_a$ and $b \geq c_b$,
\begin{equation}\label{eq-positive everywhere positive proba}
\mbb{P}(\til\varphi_{a,b}(u) > 0, \,\forall\,u \in \til{K}) \geq \frac{1}{2}.
\end{equation}
Indeed, it is straightforward to check by a union bound that there exist universal constants $c_a',c_b' >0$ such that for all $a\geq c_a'$ and $b \geq c_b'$
\[
\mbb{P}\left(\varphi_{a,b}(u) \geq \frac{h_{a,b}(u)}{2}, \,\forall\, u \in V\right) \geq \frac{3}{4}.
\]
Next, recall from the introduction that given $\mc{F}_V$, for each segment $e(u,v) \subset \til{K}$ with endpoints $u,v \in V \cup \partial \til{K} \cup \{w\}$ the restriction of $\til\varphi_{a,b}$ to $e(u,v)$ is a Brownian bridge of length $2|u-v|$ of a Brownian motion with variance 2 at time 1. Thus, applying \cite[Formula 1.3.8]{BrownianMotionHandbook} we see the following holds almost surely on $\{\til\varphi_{a,b}(u) \geq h_{a,b}(u)/2, \,\forall\,u \in V\}$,
\[
\mbb{P}(\til\varphi_{a,b}(x) > 0, \,\forall\, x \in e(u,v) \mid \mc{F}_V) = 1 - e^{-\frac{1}{2|u-v|}\varphi_{a,b}(u)\varphi_{a,b}(v)} \geq 1 - e^{-\frac{1}{8}h_{a,b}(u)h_{a,b}(v)}.
\]
Therefore, there exist universal constants $c_a'',c_b'' >0$ such that for all $a\geq c_a''$ and $b \geq c_b''$
\[
\mbb{P}\left(\til\varphi_{a,b}(u) > 0, \,\forall\,u \in \til{K} \mid \varphi_{a,b}(u) \geq \frac{h_{a,b}(u)}{2}, \,\forall\, u \in V\right) \geq \frac{2}{3}.
\]
Combining this with the second-to-last display and letting $c_j = c_j' \vee c_j''$ for $j \in \{a,b\}$ gives \eqref{eq-positive everywhere positive proba}. Note that $\var[\til\varphi_{a,b}(v)] \leq c\log(r+2)$ for some universal constant $c > 0$, and by assumption $\E[\til\varphi_{a,b}(v)] = h_{a,b}(v) \leq a \log(r+2) + b + R(a)$ so it follows that
\[
\E[\til\varphi_{c_a,c_b+\lambda}(v) \mid \til\varphi_{c_a,c_b+\lambda}(u) > 0, \,\forall\,u \in \til{K}] \leq g(r) + \lambda,
\]
for some universal (continuous, increasing) function $g$ (which can be calculated explicitly in terms of $c$, $c_a$, and $c_b$).
\end{proof}
Finally, combining the three lemmas we can prove the following corollary, which gives a high probability bound on $\E[Y \mid \mc F_{I^+}, \mathfrak E]$.
\begin{corollary}\label{cor-Conditional expectation}
There exists a constant $\Delta_1 = \Delta_1(\lambda) > 0$ and a universal function $r_1$ which satisfies $r_1(x) \to 0$ as $x \to 0$, such that the following holds
\[
\mbb{P}(\E[Y \mid \mc{F}_{I^+}, \mathfrak E] \leq - \Delta_1 \Hm(U,I) \mid \mathfrak E) \geq 1 - r_1(\xi(U,I)).
\]
\end{corollary}
Before proving this, we show how it implies Proposition~\ref{Entropic repulsion result}. Let $\Delta = \Delta_1/2$ and $X = \mbb{P}(Y \leq - \Delta \Hm(U,I) \mid \mathfrak E, \mc{F}_{I^+})$. We have
\[
\mbb{P}(Y \leq -\Delta \Hm(U,I) \mid \mathfrak E) = \E[X \mid \mathfrak E].
\]
Additionally, we have by Chebyshev's inequality and Corollary~\ref{cor-Conditional variance} that there exists a function $r_2$ such that $r_2(x) \to 0$ as $x \to 0$ and
\[
X \geq 1 - r_2(\xi(U,I)), \quad \text{a.s. on } \{\E[Y \mid \mc{F}_{I^+}, \mathfrak E] \leq - \Delta_1 \Hm(U,I)\}.
\]
It follows that
\[
\mbb{P}(Y \leq -\Delta \Hm(U,I) \mid \mathfrak E) \geq 1 - r_1(\xi(U,I)) - r_2(\xi(U,I)).
\]
\begin{proof}[Proof of Corollary~\ref{cor-Conditional expectation}]
We have
\[
\E[Y \mid \mc F_{I^+}, \mathfrak E] = \sum_{v \in D} \Hm(U,v;I) \E[\phi(v)\mid \mc F_{I^+}, \mathfrak E].
\]
Let $\mathfrak E^+ = \{\phi(v) \geq 0\,:\, v \in I^+\}$. By Lemma~\ref{lm-Monotonicity to conditioning}, we have for any $v \in D$,
\[
\E[\phi(v)\mid \mc F_{I^+}, \mathfrak E] \leq \E[\phi(v) \mid \mc{F}_{I^+}, \mathfrak E^+, \phi(v)<0 ] \leq
	\E[\phi(v)\wedge 0 \mid \mc{F}_{I^+}, \mathfrak E^+].
\]
Note that given $\mc{F}_{I^+}$, $\phi(v)$ has a normal distribution with variance at least 1. Therefore, letting $m_v= \E[\phi(v) \mid \mc{F}_{I^+}, \mathfrak E^+]$, there exist universal constants $a,b > 0$ such that
\[
\E[\phi(v)\wedge 0 \mid \mc{F}_{I^+}, \mathfrak E^+] \leq -a e^{-b m_v^2}.
\]
Thus, it suffices to show that there exists a constant $c > 0$ such that the following holds
\begin{equation}\label{eq-bound on positive boundary}
\mbb{P}\left(\sum_{v \in D} \Hm(U,v;I)m_v \leq c\Hm(U,I) \mid \mathfrak E\right) \geq 1 - r_1(\xi(U,I)).
\end{equation}
Indeed, by Markov's inequality $\sum_{v \in D} \Hm(U,v;I)m_v \leq c\Hm(U,I)$ implies
\[
\sum_{v \in D} \Hm(U,v;I) \one_{m_v \leq 2c} \geq \frac{1}{2}\Hm(U,I),
\]
and consequently
\[
\E[Y \mid \mc F_{I^+}, \mathfrak E] \leq -\frac{a e^{-4b c^2}}{2}\Hm(U,I).
\]
Turning to the proof of \eqref{eq-bound on positive boundary}, let $Y'$ be the following random variable
\[
Y' = \sum_{w \in D'} \Hm(U,w; I^+)\phi(w).
\]
It follows from the definition of the harmonic measure that
\[
\sum_{v \in D}\Hm(U,v;I)m_v - Y' = \sum_{v \in D}\Hm(U,v;I)\sum_{w \in \partial \til{K} \setminus I^+}\Hm(v,w;I^+)\til\phi(w),
\]
where for $w \in \partial \til{K}$, $\Hm(v,w; I^+)$ is the probability that a metric graph Brownian motion started at $v$ exits $\til{K}$ through $w$ without hitting $I^+$. In particular, we have
\[
\sum_{v \in D}\Hm(U,v;I)m_v \leq Y' + \lambda \Hm(U,I).
\]
Therefore, to prove \eqref{eq-bound on positive boundary} it suffices to show
\[
\mbb{P}\left( Y' \leq c\Hm(U,I) \mid \mathfrak E\right) \geq 1 - r_1(\xi(U,I)).
\]
We show this by bounding the conditional mean and variance of $Y'$. First, by Lemma~\ref{lm-Monotonicity to conditioning}, we can replace $\mathfrak E$ by $\mathfrak E' = \{\phi(v) = 0,\, v \in D\} \cap \mathfrak E^+$. Next, Lemma~\ref{lm-Conditional variance} and a straightforward adaption of the proof of Corollary~\ref{cor-Conditional variance} give
\[
\var[Y' \mid \mathfrak E'] \leq \var[Y' \mid \phi(v) = 0,\, v \in D] \leq 4\xi(U,I)\Hm(U,I)^2.
\]
Finally, Lemma~\ref{lm-Conditional expectation} and Lemma~\ref{lm-Monotonicity to conditioning} (and the fact that every $w \in D'$ is adjacent to a vertex in $D$) give
\[
\E[\phi(w) \mid \mathfrak E'] \leq g(1) + \lambda, \quad \forall w \in D'.
\]
The conclusion then follows easily.
\end{proof}

\subsection{Zero boundary case}
\label{subsec::DGFF_zerobc}
In this section we prove \eqref{DGFF zero boundary result}. We let $M$ be an exploration martingale, as introduced in Section~\ref{subsec::pre_explorationmart}, corresponding to an exploration on $E^{\geq 0}_\delta$ from $\mc I_0 = [c_\delta,d_\delta)$ with $U = [a'_\delta,b'_\delta)$. Here the exploration is on $E^{\geq 0}_\delta$, but with the understanding that we do not explore any vertices on $\partial V_\delta \setminus \mc I_0$ (equivalently, the exploration is on $E^{> 0}_\delta$ with the understanding that $\mc I_1 = [c'_\delta,d'_\delta)$). With this setup, we have $M_0 = 0$. The main step of the proof consists of bounding $\xi_k = \xi(U, \mc I_k)$. In particular, we claim that there exists $c > 0$ such that
\begin{equation}\label{eq-Hm anti-concentration zero boundary}
\sup\{\xi_k \,:\, k \geq 1\} \leq \frac{c}{|\log(\delta)|}, \quad \text{a.s.}
\end{equation}
Before proving \eqref{eq-Hm anti-concentration zero boundary}, we show how it implies \eqref{DGFF zero boundary result}. To simplify notation, we will write $\{D^{> 0}(\mc I_0, U) = \infty\}$ for the event that $[a'_\delta,b'_\delta)$ is not connected to $[c'_\delta,d'_\delta)$ in $E^{>0}_\delta$. Assuming \eqref{eq-Hm anti-concentration zero boundary}, Proposition~\ref{Entropic repulsion result}  implies that there exists $\Delta > 0$ such that for any $\epsilon > 0$, there exists $\delta_0 = \delta_0(\epsilon,L) > 0$ such that 
\begin{equation}\label{eq:: dgff zero boundary repulsion}
\mbb{P}(M_\infty \leq - \Delta \Hm(U,\mc I_\infty) \mid D^{> 0}(\mc I_0, U) = \infty) \geq 1 - \epsilon, \quad \forall \delta \leq \delta_0.
\end{equation}
To conclude, we need to upper bound the quadratic variation. Recall that $\Hm_t$ is the harmonic measure on $\mc I_t \cup \partial V_\delta$. By \eqref{eq::QV form}
\begin{equation}\label{eq:: dgff zero boundary QV bound}
\langle M \rangle_t = \sum_{v \in \mc I_t}\Hm_t(U,v) G(v,U) \leq 4 \Hm(U,\mc I_t),
\end{equation}
where we have used the fact that a random walk started on $U$ has probability at least $1/4$ of hitting $\partial V_\delta$ before returning to $U$, which implies $G(v,U) \leq 4$ for all $v \in \til V_\delta$.

Finally, by Lemma~\ref{lm-Box Hm lower bound}, $\Hm(U,\mc I_0) \geq c$. Combining this with \eqref{eq:: dgff zero boundary repulsion} and \eqref{eq:: dgff zero boundary QV bound} and applying Theorem~\ref{Time-change theorem}, we see that there exists $\epsilon = \epsilon(L) > 0$ such that
\[
\mbb{P}(\exists t\geq 0 \text{ s.t. } M_t \leq -\Delta \Hm(U,I_t)) \leq \mbb{P}\left(\exists t \geq 0 \text{ s.t. } B_t \leq -\Delta\left(c \vee \frac{t}{4}\right)\right) \leq 1 - 2\epsilon. 
\]
Therefore
\[
\mbb{P}(D^{> 0}(\mc I_0, U) = \infty) \leq 1 - \epsilon, \quad \forall \delta\leq \delta_0.
\]

\begin{proof}[Proof of \eqref{eq-Hm anti-concentration zero boundary}]
Let $D_k = \{v \in \mc I_k \,:\, \Hm_k(U,v) > 0\}$, $\mc I_k' = \mc I_k \setminus D_k$, and $\Hm_k'$ be the harmonic measure on $\mc I_k' \cup \partial V_\delta$. Let $S$ be a random walk on $\delta \mbb{Z}^2$ killed on $\partial V_\delta$, $\tau_k= \min\{n \geq 1 \,:\, S_n \in U \cup \mc I'_k\}$, and $G'_k$ be the Green's function of the random walk killed on $\partial V_\delta \cup \mc I'_k$. We have for $u \in U$ and $v \in \partial\mc I'_k$
\[
\Hm_k(u',v) = \sum_{u \in U} G'_k(u',u) \mbb{P}_{u}(S_{\tau_k} = v).
\] 
By reversibility, $\mbb{P}_{u}(S_{\tau_k} = v) = \mbb{P}_v(S_{\tau_k} = u)$. Noting also that $G_k'$ is symmetric, summing over $u'$ gives
\[
\Hm_k(U,v) = \sum_{u \in U} \mbb{P}_v(S_{\tau_k} = u)G'_k(u,U) \leq 4\mbb{P}_v(S_{\tau_k} \in U),
\]
where we have used the fact that $G'_k(u,U) \leq G(u,U) \leq 4$. There exists a universal constant $p > 0$ such that for any integer $n\geq 0$ a random walk started at $v$ will complete a loop around $B(v,2^n)$ before exiting $B(v,2^{n+1})$ (here $B(v,r)$ is the open Euclidean ball of radius $r$ around $v$). Since any loop around $v \in \mc I'_k$ that is contained in $V$ must intersect $\mc I'_k$, it follows that there exist $A, \alpha > 0$ such that the following holds for all $k \geq 1$ and $v \in \partial \mc I'_k$,
\[
\mbb{P}_v(S_{\tau_k} \in U) \leq A \left(\frac{\delta}{\dist(v, U)}\right)^\alpha.
\]
It remains to lower bound $\Hm(U, \mc I_k)$ in terms of $\dist(\mc I_k, U)$. We have already noted that there exists $c > 0$ such that $\Hm(U,\mc I_k) \geq \Hm(U, \mc I_0) \geq c$ for all $k$. On the other hand, a simple adaptation of the proof of \eqref{eq-path to boundary Harmonic measure} shows that there exists a constant $c' > 0$ such that whenever $\dist(\mc I_k, U) \leq (L \wedge 1)/2$ we have
\[
\Hm(U, \mc I_k) \geq c' \log\left(\frac{L \wedge 1}{\dist(U,\mc I_k) + \delta}\right).
\]
Combining the lower bounds on $\Hm(U, \mc I_k)$ with the upper bound on $\Hm_k(U,v)$ gives the desired conclusion.
\end{proof}

\subsection{Alternating boundary case}
\label{subsec::DGFF_alternatingbc}
In this section, we prove~\eqref{Alternating boundary difference result}. We recall the inclusion
\[
\left\{ [a_\delta, b_\delta) \stackrel{\til\phi_\delta \geq 0}{\longleftrightarrow} [c_\delta,d_\delta)\right\} \subset
	\left\{ [a_\delta, b_\delta) \stackrel{\phi_\delta \geq 0}{\longleftrightarrow}[c_\delta,d_\delta) \right\}.
\]
Therefore, it suffices to find an event $\mathfrak E$ such that
\[
\mathfrak E \subset
	\left\{ [a_\delta, b_\delta) \stackrel{\phi_\delta \geq 0}{\longleftrightarrow} [c_\delta,d_\delta)\right\} \setminus
	\left\{ [a_\delta, b_\delta) \stackrel{\til\phi_\delta \geq 0}{\longleftrightarrow}[c_\delta,d_\delta) \right\},
\]
and $\mbb{P}(\mathfrak E) \geq c$.

As usual, we provide a sketch of the proof here, leaving the technical arguments to the end of the section. By \eqref{Control on negative cluster} we can assume that $\til{\AB}_{\delta,0} \cap \Pi = \emptyset$ and work conditional on $\mc{F}_{\til{\AB}_{\delta,0}}$. We let $\mbb{P}^+$ denote the law of $\til{\phi}_{\delta}$ given $\mc{F}_{\til{\AB}_{\delta,0}}$. The first step of the proof is to upper bound the probability that $\til{\phi}_{\delta}$ contains a positive horizontal crossing. We claim that there exists a constant $c > 0$ such that the following holds almost surely on $\{\til{\AB}_{\delta,0} \cap \Pi = \emptyset\}$
\begin{equation}\label{Control on positive cluster}
\mbb{P}^+\left([a_\delta,b_\delta) \stackrel{\til \phi_\delta \geq 0}{\centernot\longleftrightarrow}[c_\delta,d_\delta)\right) \geq c.
\end{equation}
Consequently, we assume $[a_\delta,b_\delta)$ is not connected to $[c_\delta,d_\delta)$ and work conditionally on $\mc{F}_{\til{\AB}_{\delta,0} \cup \til{\A}^l_{\delta,0}}$. We let $\mbb{P}^{0,l}$ be the law of $\til{\phi}_{\delta}$ given $\mc{F}_{\til{\AB}_{\delta,0} \cup \til{\A}^l_{\delta,0}}$ and $\E^{0,l}$ be the expectation with respect to $\mbb{P}^{0,l}$. We claim that there exists a further constant $c' > 0$ such that the following holds almost surely on $\{\til{\AB}_{\delta,0} \cap \Pi = \emptyset\} \cap \left\{[a_\delta,b_\delta) \stackrel{\til \phi_\delta \geq 0}{\centernot\longleftrightarrow}[c_\delta,d_\delta)\right\}$
\begin{equation}\label{Discrete GFF RSW}
\mbb{P}^{0,l} \left( \til{\A}_{\delta,0}^l \stackrel{\phi_\delta \geq 0}{\longleftrightarrow}[c_\delta,d_\delta)\right) \geq c'.
\end{equation}
This concludes the proof. We now turn to proving the two technical claims.

\subsubsection{Proof of \eqref{Control on positive cluster}}
As in Section~\ref{subsec::alternating MGFF proof}, we let $\til{V}_\delta^+$ be the connected component of $\til{V}_\delta \setminus \til{\AB}_{\delta,0}$ containing $[a_\delta,b_\delta)$ and $[c_\delta,d_\delta)$, $\Hm^+$ be the harmonic measure on $\til V_\delta^+$, and $G^+$ be the Green's function of a metric graph Brownian motion killed on $\partial\til V_\delta^+$. The proof proceeds by analyzing an exploration martingale $M$, as introduced in Section~\ref{subsec::pre_explorationmart}, corresponding to an exploration of $\til E^{\geq 0}_\delta$ from $\mc I_0 = [a_\delta,b_\delta) \cup [c_\delta,d_\delta)$ with observable $X_U$ corresponding to the following set $U$. For each $v \in \partial \til{\AB}_{\delta,0}$, we let $\eta_v = \dist(v, \delta\mbb{Z}^2)$. Since $|\partial \til{\AB}_{\delta,0}| < \infty$ and $\partial \til{\AB}_{\delta,0} \cap V_\delta = \emptyset$ almost surely, we have
\[
\eta = \min\{\eta_v \,:\, v \in \partial\til{\AB}_{\delta,0}\} > 0, \quad \text{a.s.}
\]
Then, we let $U$ be the following set
\[
U = \left\{u \,:\, u \in \til V_\delta^+,\, \Big|\Re(u) - \frac{L}{2}\Big| \leq \frac{3 L}{8},\, \dist(u, \til{\AB}_{\delta,0}) = \frac{\eta}{2}\right\}.
\]
The following processes will be useful in the analysis
\[
\pi_t = \Hm^+_t(U,\til{\AB}_{\delta,0}), \quad \mu_t = \Hm^+_t(U,\mc I_t),
\]
where as usual $\Hm^+_t$ is the hermonic measure on $\mc I_t \cup \partial \til V_\delta^+$. We have $M_t \geq 0$ for all $t$ and $M_0 = \lambda \mu_0$, so in particular $ M_t - M_0 \geq -\lambda \mu_0$. To conclude the proof, we need to upper bound $\mu_0$ and lower bound the quadratic variation of $M$. In particular, we claim that there exist constants $c,c' > 0$ such that the following holds
\begin{align*}
\mu_0 &\leq c \frac{\eta}{\delta},\\
\langle M\rangle_\infty &\geq c' \left(\frac{\eta}{\delta}\right)^2, \quad \text{a.s. on } \left\{[a_\delta,b_\delta) \stackrel{\til \phi_\delta \geq 0}{\longleftrightarrow}[c_\delta,d_\delta)\right\}.
\end{align*}
Before proving the claim, we note that it implies by an application of Theorem~\ref{Time-change theorem} and Proposition~\ref{prop::Hitting times with drift}
\[
\mbb{P}\left([a_\delta,b_\delta) \stackrel{\til \phi_\delta \geq 0}{\longleftrightarrow}[c_\delta,d_\delta)\right) \leq \mbb{P}\left(\inf_{0 \leq t \leq c'}B_t \geq - c\lambda\right) \leq 1 -c''.
\]
For the proof, consider the electric network with vertex set $W = \mbb{Z}^2 \cup U \cup \partial \til V_\delta^+$ where two vertices $u,v \in W$ are connected with an edge of conductance $C_{u,v} = (4|u-v|/\delta)^{-1}$ if there exists a path in $\til{ \mbb{Z}}^2$ connecting $u$ to $v$ which does not contain any other points in $W$. For this section, we will write $u\sim v$ if $u$ and $v$ are connected in this network, and we let $C_u = \sum_{v \sim u} C_{u,v}$. Note that $G^+$ (as defined above) is the Green's function of the continuous time simple random walk on $W$ killed on $\partial \til V_\delta^+$. To simplify some definitions, we take $S$ to be a discrete-time simple random walk on $W$ (not killed on $\partial \til V_\delta^+$).

We now turn to proving the upper bound on $\mu_0$. Let $\tau_0 = \min\{n \geq 1 \,:\, S_n \in U \cup \partial \til V_\delta^+\}$ be the hitting time of $U \cup \partial \til V_\delta^+$. We have by a last exit decomposition
\begin{align*}
\Hm^+(u, \mc I_0) &= \sum_{u' \in U}\sum_{v \in \mc I_0} C_{u'}G^+(u,u')\mbb{P}_{u'}(S_{\tau_0} = v),\\
	&= \sum_{u' \in U}\sum_{v \in \mc I_0} C_v G^+(u',u) \mbb{P}_v(S_{\tau_0}=u')
\end{align*}
where we have used the fact that $C_{u'}\mbb{P}_{u'}(S_{\tau_0} = v) = C_v\mbb{P}_v(S_{\tau_0}=u')$ and $G^+(u,u') = G^+(u',u)$. Summing over $u \in U$ gives
\[
\mu_0 = \sum_{u' \in U}\sum_{v \in \mc I_0} C_v\mbb{P}_v(S_{\tau_0} = u')G^+(u',U) \leq \left[\max_{u' \in U}G^+(u',U)\right] \sum_{v \in \mc I_0} C_v\mbb{P}_v(S_{\tau_0} \in U).
\]
It is straightforward to show that $G^+(u',U) \leq 2\eta/\delta$ for all $u' \in U$, and it is clear that $C_v = 1$ for all $v \in \mc I_0$ such that $\mbb{P}_v(S_{\tau_0} \in U) > 0$. Further, we note that $\mbb{P}_v(S_{\tau_0} \in U) \leq c \delta$ since it is bounded by the probability that a random walk on $\delta \mbb Z$ (started from the origin) reaches $L/8$ before returning to zero. It follows that $\mu_0 \leq c \eta/\delta$ as claimed.

Next, we turn to proving the lower bound on the quadratic variation. Begin by noting that if there exists a horizontal crossing, then $\mc I_\infty$ contains a nearest neighbor path $\gamma \subset  V_\delta^+$ crossing the following strip
\[
\Pi' = \left\{z \,:\, \Big|\Re(z) - \frac{L}{2}\Big| \leq \frac{L}{4}\right\}.
\]
We claim that for any such $\gamma$
\[
\var^+(X_U) - \var^+(X_U \mid \mc{F}_\gamma) = \sum_{v \in \gamma} \Hm^+(U,v;\gamma) G^+(v,U) \geq c \left(\frac{\eta}{\delta}\right)^2.
\]
For the proof, note that there exists $c > 0$ such that $\Hm^+(v,U) \geq c$ for all $v \in \gamma$, where $c$ is independent of $\gamma$ and $v$. This follows from the fact that a random walk started in $\Pi'$ will exit $V_\delta$ before exiting the strip $\{z \,:\, |\Re(z) - L/4| \leq 3L/8\}$ with probability bounded uniformly away from zero. Such a walk will necessarily hit $U$ before $\partial \til V_\delta$. Next, it is straightforward to show that $G^+(u,U) \geq G^+(u,u) \geq \eta/\delta$ for all $u \in U$, so we conclude $G^+(v,U) \geq c \eta/\delta$,  Consequently
\[
\var^+(X_U) - \var^+(X_U \mid \mc{F}_\gamma)  \geq c \frac{\eta}{\delta} \Hm^+(U,\gamma).
\] 
To lower bound the harmonic measure, the idea is to replace $\gamma$ with a line. More precisely, we take
\[
\gamma^* = \left\{w \in V_\delta \,:\, \Im(w) \in \left[\frac{1}{2}, \frac{1}{2} + \delta\right), \, \Big| \Re(w) - \frac{L}{2} \Big| \leq \frac{L}{4}\right\},
\]
and note that there exists a universal constant $c > 0$ such that for $w \in \gamma^*$, $\Hm^+(w,\gamma) \geq c$. This follows from the fact that a random walk started at $w$ will hit the line $\{z \,:\, \Re(z) = L/2\}$ before $\partial \til V_\delta^+$, and then exit $V_\delta$ before exiting $\Pi'$ is bounded uniformly away from 0. Therefore, we have $\Hm^+(U,\gamma) \geq c\Hm^+(U,\gamma^*)$. To bound $\Hm^+(U,\gamma^*)$ from below, we proceed by a last-exit decomposition. Let $\tau = \min\{n \geq 1\,:\, S_n \in U \cup \partial \til V_\delta^+ \cup \gamma^*\}$. By the same argument given in the proof of the upper bound on $\mu_0$,
\[
\Hm^+(U,\gamma^*) = \sum_{w \in \gamma^*}\sum_{u \in U} C_w \mbb{P}_w(S_{\tau} = u)G^+(u,U) \geq \frac{\eta}{\delta} \sum_{w \in \gamma^*}\mbb{P}_w(S_\tau \in U),
\]
where we have used the fact that $C_w = 1$ for all $w \in \gamma^*$ and $G^+(u,u) \geq \eta/\delta$ for all $u \in U$. Finally, by Lemma~\ref{lm-Box Hm lower bound} we have
\[
\sum_{w \in \gamma^*} \mbb{P}_w(S_\tau \in U) \geq c.
\]
This concludes the proof.

\subsubsection{Proof of \eqref{Discrete GFF RSW}}
Let $\til V_\delta^{l,0}$ be the connected component of $\til V_\delta \setminus(\til{\A}_{\delta,0}^l \cup \til{\AB}_{\delta,0})$ containing $[c_\delta,d_\delta)$, and $V_\delta^{l,0} = \til V_\delta^{l,0} \cap \delta \mbb{Z}^2$. Let $W$ be the following set of vertices
\[
W= \{v \in V_\delta^{l,0} \,:\, \dist(v, \til{\A}_{\delta,0}^l) \leq \delta\}.
\]
We note that 
\[
\left\{[c_\delta,d_\delta) \stackrel{\phi_\delta \geq 0}{\longleftrightarrow} \til{\A}_{\delta,0}^l \right\} = 
\left\{[c'_\delta,d'_\delta) \stackrel{\phi_\delta \geq 0}{\longleftrightarrow} W \right\}.
\]
Consequently, we would like to explore $E^{\geq 0}_\delta \cap V_\delta^{l,0}$ from $W$, but in order to apply Proposition~\ref{Entropic repulsion result} we need to start our exploration ``one step'' away from the boundary. That is, letting $\mc I_0$ and $U$ be given by
\begin{align*}
\mc I_0 &= \left\{v \in V_\delta^{l,0} \,:\, \delta < \dist(v, \til{\A}_{\delta,0}^l) \leq 2\delta\right\},\\
U &= [c'_\delta,d'_\delta)\cap\left\{z \,:\, \Big|\Im(z) - \frac{1}{2}\Big| \leq \frac{1}{8}\right\},
\end{align*}
we will show that
\begin{equation}\label{eq-Rough boundary near crossing}
\mbb{P}^{l,0} \left(U \stackrel{\phi_\delta \geq 0}{\longleftrightarrow} \mc I_0 \right) \geq c.
\end{equation}
Assuming this result for now, we show how to conclude the proof of \eqref{Discrete GFF RSW}. Recall that $\mc G_{V_\delta^{l,0} \setminus W}$ is the $\sigma$-field generated by $\{\sign(\phi_\delta(v)) \,:\, v \in V_\delta^{l,0} \setminus W\}$ and note that we have
\[
\left\{U \stackrel{\phi_\delta \geq 0}{\longleftrightarrow} \mc I_0\right\} \in \mc G_{V_\delta^{l,0} \setminus W}.
\]
Note also that if $U$ is connected to $\mc I_0$, there exists a random vertex $w^* \in W$ such that if $\phi_\delta(w^*) > 0$, then $U$ is connected to $W$ (and thus to $[a_\delta,b_\delta)$). Note that $w^*$ is measurable with respect to $\mathcal G_{V_\delta^{l,0} \setminus W}$. Therefore, it suffices to show that there exists a universal constant $c > 0$ such that for all $w \in W$ 
\[
\mbb{P}^{l,0}\left(\phi_\delta(w) > 0 \mid \mc G_{V_\delta^{l,0} \setminus W}\right) \geq c \quad \text{a.s.}
\]
It will then follow that
\[
\mbb{P}^{l,0}\left(\phi_\delta(w^*) > 0 \mid \mc G_{V_\delta^{l,0}},\, U \stackrel{\phi_\delta \geq 0}{\longleftrightarrow} \mc I_0\right) \geq c \quad \text{a.s.}
\]
By Lemma~\ref{lm-Monotonicity to conditioning}, it suffices to show that for $\mathfrak E = \{\phi_\delta(v) < 0,\, \forall v \in V_\delta^{l,0}\setminus W\}$,
\[
\mbb{P}^{l,0}\left(\phi_\delta(w) > 0 \mid \mathfrak E\right)  \geq c, \quad \forall w \in W.
\]
The proof is as follows. By the Markov property of the GFF, we have for any $w \in W$,
\begin{align*}
\phi_\delta(w) = Y_w + Z_w, 
\end{align*}
where $Z_w$ is normal, mean zero, and independent of $\mc F_{V_\delta^{l,0} \setminus W}$, and $Y_w = \E^{l,0}[\phi_\delta(w) \mid \mc{F}_{V_\delta^{l,0} \setminus W}]$. It is straightforward to show that there exist $c,c' > 0$ such that
\[
c \frac{\dist(w, \til{\A}_{\delta,0}^l)}{\delta} \leq \var^{l,0}[Z_w] \leq \var^{l,0}[\phi_\delta(w)] \leq c' \frac{\dist(w,\til{\A}_{\delta,0}^l)}{\delta}.
\]
Finally, by Lemma~\ref{lm-Conditional expectation} there exists $c'' >0$ such that
\begin{align*}
\E^{l,0}\left[Y_w \mid \mc{G}_{V_\delta^{l,0} \setminus W}\right] &\geq -c'' \frac{\dist(w,\til{\A}_{\delta,0}^l)}{\delta}.%,\\
%\var^{l,0}[Y_w \mid \mc{G}_{V_\delta^{l,0} \setminus W}] \leq \var^{l,0}[Y_w] &\leq c' \var^{l,0}[Z_w].
\end{align*}
Combining these bounds we obtain $\mbb{P}(\phi_\delta(w) > 0 \mid \mathfrak E) \geq c$ as promised. This concludes the proof.

\begin{proof}[Proof of \eqref{eq-Rough boundary near crossing}]
We assume without loss of generality that $\dist(U,\mc I_0) \geq 10\delta$. Otherwise the probability of connection is lower bounded by $2^{-20}$, say.

As usual, we consider an exploration martingale $M$ corresponding to an exploration of $E^{\geq 0}_\delta \cap V_{\delta}^{l,0}$ from $\mc I_0$ with observable $X_U$, with $U$ as above. The goal is to bound both the value of the martingale and its quadratic variation. As usual, we introduce the following processes
\begin{align*}
\pi_t = \Hm^{l,0}_t(U,[c_\delta,d_\delta)), \quad \mu_t = \Hm^{l,0}_t(U,\mc I_t).
\end{align*}
We will use Proposition~\ref{Entropic repulsion result} to show that there exists $c > 0$ such that for any $\epsilon$, there exists $\delta_0 = \delta_0(\epsilon) > 0$ such that for $\delta \leq \delta_0$,
\begin{equation}\label{eq-Entropic repulsion rough to straight}
\mbb{P}\left(M_\infty - M_0 \leq -c\mu_0 - \lambda(\pi_0 - \pi_\infty)  \mid D^{\geq 0}_\delta(\mc I_0, U) = \infty\right) \geq 1 - \epsilon.
\end{equation}
Postponing the proof of this claim until the end of the section, we show how it can be used to obtain the desired result. Fist, we bound the quadratic variation as follows
\begin{align*}
\langle M \rangle_t &= \sum_{v \in \mc I_t} \Hm^{l,0}_t(U,v)G^{l,0}_0(v,U) \\
&\leq 16\sum_{v \in \mc I_t}\Hm^{l,0}_t(U,v)\Hm_0^{l,0}(v,[c_\delta,d_\delta))\\
&=16( \pi_0 - \pi_t)
\end{align*}
where we used the fact that a random walk started on $U$ hits $[c_\delta,d_\delta)$ before returning to $U$ with probability at least $1/4$ to obtain 
\[
G^{l,0}_0(v,U) \leq 4\Hm^{l,0}_0(v,U) \leq 16 \Hm^{l,0}_0(v,[c_\delta,d_\delta)).
\]
Next, we note that by Lemma~\ref{lm-Box Hm lower bound} we have for some $c > 0$
\[
\mu_0 \geq \Hm^+(U, [a_\delta,b_\delta)) \geq c.
\]
Thus, for $\epsilon >0$ small enough and $\delta < \delta_0$,
\[
\mbb{P}^{l,0}\left(U \stackrel{\phi_\delta \geq 0}{\centernot\longleftrightarrow} \mc I_0 \right)\leq \mbb{P}\left(\inf \left\{B_t + \frac{\lambda}{16}t\,:\, t \geq 0\right\} \leq -c\right) + \epsilon \leq 1 - c'.
\]
Turning to the proof of \eqref{eq-Entropic repulsion rough to straight}, we begin by upper bounding $M_\infty$. To this end, we let $\xi_k = \xi(U,\mc I_k)$ and note as in \eqref{eq-Hm anti-concentration zero boundary} that there exists $c > 0$ such that
\[
\sup\{\xi_k \,:\, k \geq 0\} \leq \frac{c}{|\log(\delta)|}.
\]
The proof is the same as that of \eqref{eq-Hm anti-concentration zero boundary} so we omit further details. We can then apply Proposition~\ref{Entropic repulsion result} with
\[
I^+ = \{v \in \mc I_\infty \,:\, \sign(\phi_\delta(v)) = 1\}, \quad I^- = \{v \in \mc I_\infty \,:\, \sign(\phi_\delta(v)) = -1\},
\]
and obtain that there exists $\Delta = \Delta(\lambda) > 0$ such that for any $\epsilon > 0$ there exists $\delta_1 > 0$ such that for $\delta \leq \delta_1$
\[
\mbb{P}^{l,0}\left(M_\infty \leq - \Delta \mu_\infty + \lambda \pi_\infty \mid U \stackrel{\phi_\delta \geq 0}{\centernot\longleftrightarrow} \mc I_0 \right) \geq 1 - \epsilon.
\]
Next, we lower bound $M_0$. Trivially, $\E^{l,0}[M_0] = \lambda \pi_0$. For the variance, we have
\[
\var^{l,0}[M_0] = \sum_{v\in \mc I_0}\Hm^{l,0}_0(U,v)G^{l,0}(v,U) \leq 64\xi_0 \mu_0^2 \leq \frac{c}{|\log(\delta)|}\mu_0^2,
\]
where we have used the facts 
\begin{align*}
G^{l,0}(U,\mc I_0) &\leq 16 \Hm^{l,0}(U,\mc I_0),\\
\Hm_0^{l,0}(U,v) &\leq 4 \xi_0 \Hm_0^{l,0}(U,\mc I_0), \quad \forall v \in \mc I_0.
\end{align*}
Therefore, we conclude that for any $\epsilon > 0$, there exists $\delta_2$ such that for $\delta < \delta_2$,
\[
\mbb{P}^{l,0}\left(M_0 \geq - \frac{\Delta}{2}\mu_0 + \lambda \pi_0\right) \geq 1 - \epsilon.
\]
Combining the two bounds, and noting that $\mu_t$ is increasing we obtain that for any $\epsilon > 0$ and $\delta < \delta_1 \wedge \delta_2$
\[
\mbb{P}^{l,0}\left(M_\infty - M_0 \leq -\frac{\Delta}{2} \mu_0 - \lambda(\pi_0 - \pi_\infty) \mid U \stackrel{\phi_\delta \geq 0}{\centernot\longleftrightarrow} \mc I_0 \right) \geq 1 - 2\epsilon.
\]
This concludes the proof of \eqref{eq-Entropic repulsion rough to straight}.
\end{proof}

%%%%%%%%%%%%%%%%%%%%%%%%%%%%%%
\section{(Non)-Existence of closed pivotal edges for metric graph percolation}
In this section we prove Theorem~\ref{thm::Pivotal edges result}. First, we recall some definitions and notation. Let $\tilde{\phi}_{\delta}$ be a metric graph GFF on $\tilde{V}_{\delta}$ with alternating boundary condition~\eqref{eqn::alternatingbc}. Denote by $E_{\delta}$ the edge set of the nearest-neighbor graph on $V_{\delta}$. We let $\omega: E_{\delta}\to \{0,1\}$ be such that $\omega(e)=1$ if $\tilde{\phi}_{\delta}(u)\ge 0$ for all $u\in I_e$ and $\omega(e)=0$ otherwise. In the former case, we say the edge is open, in the latter that it is closed. The notations $\omega^e$ and $\omega_e$ are defined in~\eqref{eqn::def_bondperco}. An edge $e$ is pivotal if there is a horizontal crossing in $\omega^e$ but there is no such crossing in $\omega_e$. Note that the pivotality of an edge does not depend on the status of the edge itself. We say that $e$ is a closed pivotal edge (resp. open pivotal edge) if $e$ is pivotal and it is closed (resp. open). Our goal is to show that the probability that there exists a closed pivotal edge decays to 0 as $\delta \to 0$. In particular, we want to show that there exists a constant $c > 0$ (depending only on $L$) such that the following holds
\begin{equation}\label{no closed pivotals}
\mbb{P}(\text{there exists a closed pivotal edge}) \leq \frac{c}{\sqrt{|\log(\delta)|}}.
\end{equation}
By symmetry, it suffices to consider only edges $e$ on the left half of $V_\delta$. That is, if we let
\[
E_\delta^l = \left\{e \in E\,:\, \exists\, x \in I_e,\, \Re(x) \leq \frac{L}{2}\right\},
\]
it suffices to show
\[
\mbb{P}\left(\text{there exists a closed pivotal edge in $E_\delta^l$}\right) \leq \frac{c}{\sqrt{|\log(\delta)|}}.
\]
In fact, we will show that
\[
\mbb{P}\left(\text{there exists a closed pivotal edge in $E_\delta^l$}\mid \mc{F}_{\til{\A}_{\delta,0}^l}\right) \leq \frac{c}{\sqrt{|\log(\delta)|}} \quad \text{a.s.}
\]
Note that if there is a positive horizontal crossing in the metric graph (that is, $\til{\A}_{\delta,0}^l \cap [c_\delta,d_\delta) \neq \emptyset$) then there are no closed pivotal edges and that this event is measurable with respect to $\mathcal F_{\tilde A_{\delta,0}^l}$. Therefore, we assume from now on that this is not the case.

To simplify notation, we let $\mbb{P}^{l}$ denote the law of $\til{\phi}_{\delta}$ given $\mc{F}_{\til{\A}_{\delta,0}^l}$ and $\E^{l}$ denote the expectation with respect to $\mbb{P}^{l}$. Similarly, we let $\til{V}_\delta^{l}$ be the connected component of $\til{V}_\delta \setminus \til{\A}_{\delta,0}^l$ containing $[c_\delta,d_\delta)$ and $V_\delta^{l}= \til{V}_\delta^{l} \cap \delta \mbb{Z}^2$. Note that, given that there is no horizontal crossing in the metric graph, the set of closed pivotal edges consists of all edges with one endpoint in $\tilde \A_{\delta,0}^l$ and the other in $\tilde \A_{\delta,0}^r$. Therefore, given $\mathcal{F}_{\tilde A_{\delta,0}^l}$ the set of pivotal edges is increasing in $\tilde \A_{\delta,0}^r$ and therefore increasing in $\{\til\phi_\delta(v) \,:\, v \in \til V_\delta^l\}$. We will therefore assume from now on that $\til\phi_{\delta}$ has zero boundary condition on $[b_\delta,c_\delta) \cup [d_\delta, a_\delta)$.

The proof is essentially the same as that of \eqref{Control on positive cluster} and consists of analyzing an exploration martingale $M$, as introduced in Section~\ref{subsec::pre_explorationmart}, corresponding to an exploration on $\til{E}^{\geq 0}_\delta$ from $\mc I_0 = [c_\delta,d_\delta)$ with observable given by the following set. For each $v \in \partial\til{\A}_{\delta,0}^l \setminus \partial \til V_\delta$, we let $\eta_v = \dist(v, V_\delta^l)$. Since  $|\partial\til{\A}_{\delta,0}^l| < \infty$ and $\partial\til{\A}_{\delta,0}^l \cap \delta \mbb Z^2 \subset \partial \til V_\delta$ almost surely, we have
\[
\eta = \min\{ \eta_v \,:\, v \in \partial\til{\A}_{\delta,0}^l \setminus \partial \til V_\delta > 0\} \quad \text{a.s.}
\]
We then set the observable set $U$ to be
\[
U = \left\{u \in \til V_\delta^l \,:\, \dist(u,\partial \til V_\delta^l) = \frac{\eta}{2}, \, \Re(u) \leq \frac{3L}{4} \right\}.
\]
Without loss of generality, we assume that $U$ is non-empty (and indeed that there exists $u \in U$ with $\Re(u) \leq (L/2)+\delta$) since otherwise there can be no pivotal edges in $E_\delta^l$. As usual, we let $\Hm^l_t$ be the harmonic measure on $\mc I_t \cup \partial \tilde V_\delta^l$ and $\mu_t = \Hm^l_t(U, \mc I_t)$. As usual, since $\mc I_0 \subset \partial \tilde V_\delta^l$ we write $\Hm^l$ for $\Hm^l_0$. Note that $M_0 = \lambda \mu_0$ and $M_t \geq 0$ for all $t$, so $M_t - M_0 \geq - \lambda \mu_0$. We claim that there exist constants $c, c' > 0$ such that
\begin{align*}
\mu_0 &\leq c \frac{\eta}{\delta},\\
\langle M \rangle_\infty &\geq c' \left(\frac{\eta}{\delta}\right)^2 |\log(\delta)|,\quad \text{a.s. on } \{\text{there exists a closed pivotal edge in } E_\delta^l\}.
\end{align*}
Assuming this claim for now, an application of Theorem~\ref{Time-change theorem} then gives (after rescaling the Brownian motion)
\begin{align*}
\mbb{P}^l(\text{there exists a closed pivotal edge in } E_\delta^l) &\leq \mbb{P}\left(\inf_{0 \leq t \leq c' |\log(\delta)|} B_t \geq - c\lambda\right) \\
&\leq \frac{c''}{\sqrt{|\log(\delta)|}}.
\end{align*}
To prove the claim we introduce the electric network with vertex set $W = \delta\mbb{Z}^2 \cup U \cup \partial \til V_\delta^l$ where two vertices $u,v \in W$ are connected with an edge of conductance $C_{u,v} = (4|u-v|/\delta)^{-1}$ if there exists a path in $\delta\til{ \mbb{Z}}^2$ connecting $u$ to $v$ which does not contain any other points in $W$. For the rest of this proof, we will write $u\sim v$ if $u$ and $v$ are connected in this network, and we let $C_u = \sum_{v \sim u} C_{u,v}$. We let $G^l$ be the Green's function corresponding to the continuous time simple random walk on $W$ killed on $\partial \til V_\delta^l$, but take $S$ to be a discrete-time simple random walk on $W$ (not killed on $\partial \tilde V_\delta^l$). 

For the upper bound on $\mu_0$, we let $\tau = \min\{n \geq 1 \,:\, S_n \in U \cup \partial \til V_\delta^l\}$. By the arguments given in the proof of \eqref{Control on positive cluster},
\[
\Hm^l(U,\mc I_0) = \sum_{u \in U}\sum_{v \in \mc I_0} C_v \mbb{P}_v(S_{\tau} = u) G^l(u,U) \leq 2\frac{\eta}{\delta} \sum_{v \in \mc I_0} \mbb{P}_v(S_{\tau} \in U),
\]
where we used the facts that $C_v = 1$ for all $v \in \mc I_0$ such that $\mbb{P}_v(S_{\tau} \in U) > 0$ and $G^l(u,U) \leq 2 \eta/\delta$ for all $u \in U$. Finally, there exists $c > 0$ such that $P_v(S_{\tau} \in U) \leq c \delta$ since this probability is upper bounded by the probability that a one-dimensional simple random walk hits $L/4$ before returning to zero. Since $|\mc I_0| \geq c'/\delta$ we conclude that $\mu_0 \leq c \eta/\delta$ as promised.

To lower bound the quadratic variation, we note that if there exists a closed pivotal edge in $E_\delta^l$, then $\mc I_\infty$ contains a nearest-neighbor path $\gamma$ in $V_\delta^l$ satisfying the following conditions. First, that it crosses the following strip
\begin{equation}
\Pi' = \left\{z \,:\, \frac{L}{2} + \delta \leq \Re(z) \leq \frac{5L}{8}\right\}.
\end{equation}
Second, that there exists $v^* \in \gamma$ and $w \in \til{\A}_{\delta,0}^l \cap \delta \mbb{Z}^2$ satisfying $\Re(w) \leq L/2 + \delta$ and $|w-v^*| = \delta$. We claim that there exists $c > 0$ such that for any such path
\[
\var^l(X_U) - \var^l(X_U \mid \mc{F}_\gamma) \geq c \left(\frac{\eta}{\delta}\right)^2 |\log(\delta)|.
\]
Indeed, we have
\begin{align*}
\var^l(X_U) - \var^l(X_U \mid \mc{F}_\gamma) &= \sum_{v \in \gamma}\Hm(U,v; \gamma \cup \partial \tilde V_\delta^l)G^l(v,U), \\
&\geq c \frac{\eta}{\delta} \Hm(U,\gamma; \gamma \cup \partial \tilde V_\delta^l),
\end{align*}
where we used the fact that $\Hm^l(v,U) \geq c$ for all $v \in \gamma$ since it is lower bounded by the probability that a simple random walk on $\delta \mathbb Z$ started at $\delta \lceil 5L/8\delta \rceil$ hits 0 before $\delta \lfloor 3L/4\delta \rfloor$ and the fact that $G(u,U) \geq \eta/\delta$ for all $u \in U$. Therefore, we want to show that there exists $c > 0$ such that for any path $\gamma$ that satisfies the conditions above,
\[
\Hm(U,\gamma; \gamma \cup \partial \tilde V_\delta^l) \geq c \frac{\eta}{\delta}|\log(\delta)|.
\]
The proof is essentially identical to that of \eqref{eq-path to boundary Harmonic measure}. For $n\geq0$ an integer, we let $Q_n$ be the box of radius $r_n = 2^{-n-2}L$ centered at $v^*$, $A_n = Q_n \setminus Q_{n+1}$, $U_n = U \cap A_n$, and $\gamma_n = \gamma \cap A_n$. We note that $N = \max\{n \geq 1\,:\, r_n \geq 1000\delta\}$ satisfies $N \geq c |\log(\delta)|$. Finally, we claim that there exists a universal constant  $c > 0$ such that
\[
\Hm(U_n,\gamma_n; \gamma \cup \partial \tilde V_\delta^l) \geq c \frac{\eta}{\delta},\quad 1 \leq n \leq N\,.
\]
For the proof, we begin as usual with a last-exit decomposition. Take $1 \leq n \leq N$ and let $\tau = \min\{k \geq 1 \,:\, S_k \in U \cup \gamma \cup \partial \til V_\delta^l\}$ (recall $S$ is a random walk on $W$). We have
\begin{align*}
\Hm(U_n,\gamma_n;\gamma \cup \partial \tilde V_\delta^l) &= \sum_{u \in U_n}\sum_{u' \in U} \sum_{v \in \gamma_n} G^l(u,u')C_{u'} \mbb{P}_{u'}(S_\tau = v),\\
&=\sum_{u \in U_n}\sum_{u' \in U}\sum_{v \in \gamma_n} G^l(u',u)C_v \mbb{P}_v(S_\tau = u'),\\
&\geq \frac{\eta}{\delta}\sum_{v \in \gamma_n} \mbb P_v(S_\tau \in U).
\end{align*}
From here, the details of the proof are the same as for the proof of \eqref{eq-path to boundary Harmonic measure}. Let $v_1 \in \gamma_n$ be such that $2r_n/3 - \delta < |v_1 - v^*|_{\ell_\infty} \leq 2r_n/3$, $v_2 \in \gamma_n$ be such that $5r_n/6 \leq |v_2 - v^*|_{\ell_\infty} < 5r_n/6 + \delta$, and $Q_{n,2}$ be the box of radius $r_n/12$ centered at $v_2$. With these choices, the distance between $v_1$ and $Q_{n,2}$, and between $\{v_1\} \cup Q_{n,2}$ and $A_n^c$ are of the same order as $r_n$. Therefore, letting $\mathfrak E$ be the event that $S$ hits $Q_{n,2} \cap \gamma$ and then hits $U$ before exiting $A_n$, there exists a universal constant $c > 0$ such that $\mbb{P}_{v_1}(\mathfrak E) \geq c$. By a last exit decomposition
\[
\mbb{P}_{v_1}(\mathfrak E) \leq \sum_{v \in \gamma} G^{l,*}(v_1,v)\mbb{P}_v(S_\tau \in U_n),
\]
where $G^{l,*}(v_1,v)$ is the expected number of visits a random walk started at $v_1$ makes to $v$ after hitting $Q_{n,2}\cap \gamma$ and before exiting $A_n$. Finally, it follows easily from \cite[Theorem 4.4.4, Proposition 4.6.2]{LawlerLimic10} that $G^{l,*}(v_1,\cdot)$ is uniformly bounded. That is, there exists a universal constant $c > 0$ such that
\[
G^{l,*}(v_1,v) \leq c,\quad v \in A_n \cap \delta \mathbb{Z}^2.
\]
This concludes the proof.

%%%%%%%%%%%%%%%%%%%%%%%%%%%%%%
%\section{Expected number of pivotal edges for metric graph percolation}
%\input{tex/expectednumber}
%%%%%%%%%%%%%%%%%%%%%%%%%%%%
\section{Limits of crossing probabilities}
\label{sec::scaling_limits}
In this section, we discuss level lines of discrete GFF and metric graph GFF. 
Fix the rectangle $R_L=(0,L)\times (0,1)$ and $V_{\delta}=R_L\cap \delta\Z^2$. Recall that the four corners of $R_L$ are denoted by $a, b, c, d$ in counter-clockwise order with $b=0$, and that the four corners of $V_{\delta}$ are denoted by $a_{\delta}, b_{\delta}, c_{\delta}, d_{\delta}$ in counter-clockwise order with $b_{\delta}$ closest to the origin, i.e. $b_{\delta}=(\delta, \delta)\in\delta\Z^2$. 

Consider discrete GFF $\phi_{\delta}$ on $V_{\delta}$ with either zero boundary condition or alternating boundary condition~\eqref{eqn::alternatingbc}.  In either case, we say that the vertices on $[a_{\delta}, b_{\delta})$ have positive value and the vertices on $[b_{\delta}, c_{\delta})$  have negative value.  
The level line $\eta_{\delta}$ of $\phi_{\delta}$ is defined as follows: it starts from $b_{\delta}^{\diamond}=(0, 3\delta/2)$, lies on the dual lattice of $\delta\Z^2$ and turns at every dual-vertex in such a way that it has vertices with positive value on its left and negative value on its right. If there is an indetermination when arriving at a dual-vertex, turn left. The level lines stop when they hit the boundary segments $[c_{\delta}, d_{\delta})$ or $[d_{\delta}, a_{\delta})$. 
We will consider the convergence of $\eta_{\delta}$ in this section. We use the following metric on planar curves: suppose $\eta_1$ and $\eta_2$ are unparameterized continuous curves, then 
\[d(\eta_1, \eta_2)=\inf_{u_1, u_2}\sup_{t\in [0,1]}|\eta_1(u_1(t))-\eta_2(u_2(t))|,\]
where the inf is over increasing homeomorphisms $u_1, u_2: [0,1]\to [0,1]$. 
Using techniques in~\cite{SchrammSheffield09}, we have the following convergence of the law on $\eta_{\delta}$. 

\begin{theorem}\label{thm::crossingproba_cvg_dgff}
There exists $\lambda_0>0$ such that, when $\lambda=\lambda_0$, the law of $\eta_{\delta}$ converges weakly to $\eta\sim\SLE_4(-2; -2)$ in $R_L$ from $b$ to $d$ with force points $(a; c)$ as $\delta\to 0$. As a consequence, \eqref{eqn::crossingproba_limit} holds.
\end{theorem} 

To prove Theorem~\ref{thm::crossingproba_cvg_dgff}, we will first introduce SLE process in Section~\ref{subsec::sle} and then complete the proof in Section~\ref{subsec::scaling_limits_dgff}. We will discuss the limit of crossing probabilities in~\eqref{DGFF zero boundary result} at the end of Section~\ref{subsec::scaling_limits_dgff}. We will discuss the limit of crossing probabilities in~\eqref{MGFF alternating boundary result} in Section~\ref{subsec::scaling_limits_mgff}.

\subsection{Preliminaries on SLE}
\label{subsec::sle}

We denote by $\HH$ the upper-half plane. We call a compact subset $K$ of $\overline{\HH}$ an $\HH$-hull if $\HH\setminus K$ is simply connected. By Riemann's mapping theorem, there exists a unique conformal map $g_K$ from $\HH\setminus K$ onto $\HH$ with the normalization $\lim_{z\to\infty}|g_K(z)-z|=0$. With such normalization, we say that $g_K$ is normalized at $\infty$. 

We consider the following collections of $\HH$-hulls. First, consider families of conformal maps $(g_t, t\ge 0)$ obtained by solving the Loewner equation: for each $z\in\HH$, 
\[\partial_t g_t(z)=\frac{2}{g_t(z)-W_t},\quad g_0(z)=z,\]
where $(W_t, t\ge 0)$ is a real-valued continuous function, which we call the driving function. Second, for each $z\in\HH$, define the swallowing time $T_z$ to be \[\sup\left\{t\ge 0: \inf_{s\in[0,t]}|g_s(z)-W_s|>0\right\}.\] Finally, denote by $K_t$ the closure of $\{z\in\HH: T_z\le t\}$. Then $g_t$ is the conformal map from $\HH\setminus K_t$ onto $\HH$ normalized at $\infty$. The collection of $\HH$-hulls $(K_t, t\ge 0)$ is called a Loewner chain parameterized by the half-plane capacity. 

For $\kappa\ge 0$, Schramm Loewner Evolution, denoted by $\SLE_{\kappa}$, is the Loewner chain with driving function $W_t=\sqrt{\kappa}B_t$ where $(B_t, t\ge 0)$ is a standard one-dimensional Brownian motion. It was proved in \cite{RohdeSchrammSLEBasicProperty} that  $(K_t, t\ge 0)$ is almost surely generated by a continuous transient curve, i.e. there exists a continuous curve $\eta$ such that, for each $t\ge 0$, the set $\HH\setminus K_t$ is the unbounded connected component of $\HH\setminus \eta[0,t]$ and $\lim_{t\to\infty}|\eta(t)|=\infty$. In this section, we focus on $\kappa\in (0,4]$ when the curve is simple. 

$\SLE_{\kappa}(\rho^L; \rho^R)$ is a variant of $\SLE_{\kappa}$ process where one keeps track of two extra marked points on the boundary. 
Let $y^L\le 0\le y^R$ and $\rho^L, \rho^R\in \R$. An $\SLE_{\kappa}(\rho^L; \rho^R)$ process with force points $(y^L; y^R)$ is the Loewner chain driven by $W_t$ that solves the following system of SDEs: 
\begin{align*}
dW_t=\sqrt{\kappa}d B_t+\frac{\rho^Ldt}{W_t-V_t^{L}}+\frac{\rho^{R}dt}{W_t-V_t^{R}},\quad 
dV_t^{L}=\frac{2dt}{V_t^{L}-W_t},\quad dV_t^{R}=\frac{2dt}{V_t^{R}-W_t},  
\end{align*}
where $W_0=0, V_0^L=y^L, V_0^R=y^R$. 
It turns out that such process exists for all time if $\rho^L, \rho^R>-2$. 
If $\rho^L\le -2$ or $\rho^R\le -2$, it exists up to
the first time that the process swallows $y^L$ or $y^R$. Moreover, the process is generated by continuous curve up to and including the same time. 

The above SLE processes are defined in $\HH$, for other simply connected domains we define SLE process via conformal image. Suppose $\Omega$ is a non-trivial simply connected domain and $a, b, c, d$ are four boundary points lying on locally connected components in counterclockwise order. Then $\SLE_{\kappa}(\rho^L; \rho^R)$ in $\Omega$ from $b$ to $d$ with force points $(a; c)$ is $\varphi^{-1}(\eta)$ where $\eta$ is an $\SLE_{\kappa}(\rho^L; \rho^R)$ from $0$ to $\infty$ with force points $(y^L; y^R)$ and $\varphi$ is any conformal map from $\Omega$ onto $\R$ such that $\varphi(a)=y^L\le \varphi(b)=0\le \varphi(c)=y^R<\varphi(d)=\infty$. 

In this section, we focus on $\SLE_4(\rho^L; \rho^R)$ with $\rho^L=\rho^R=-2$ and force points $y^L<0<y^R$. Then the above SDEs become 
\begin{equation}\label{eqn::SDE_sle4}
dW_t=2dB_t+\frac{2dt}{V_t^L-W_t}+\frac{2dt}{V_t^R-W_t}, \quad dV_t^{L}=\frac{2dt}{V_t^{L}-W_t},\quad dV_t^{R}=\frac{2dt}{V_t^{R}-W_t}. 
\end{equation}
Denote by $\eta$ the continuous curve corresponding to the Loewner chain driven by $W$. 
Let $T$ be the first time that the process swallows $y^L$ or $y^R$. We will calculate the probability $\PP(\eta(T)=y^R)$. 
Set 
\[\hat{W}_t=\frac{2W_t-V_t^L-V_t^R}{V_t^R-V_t^L}.\]
Then $T=\inf\{t: \hat{W}_t\in\{-1, +1\} \}$. It\^{o}'s formula gives 
\[d\hat{W}_t=\frac{4dB_t}{V_t^R-V_t^L}.\]
In particular, $(\hat{W}_{t\wedge T}, t\ge 0)$ is a bounded martingale. Optional stopping theorem gives $\E[\hat{W}_T]=\hat{W}_0$. From there, we obtain
\begin{equation}\label{eqn::sle_crossingproba}
\PP\left(\eta(T)=y^R\right)=\PP\left(\hat{W}_T=1\right)=\frac{-y^L}{y^R-y^L}. 
\end{equation}
This gives~\eqref{eqn::crossingproba_limit} assuming the convergence of the level line in Theorem~\ref{thm::crossingproba_cvg_dgff}.  

Below, we perform a time change in the system which will be useful in Section~\ref{subsec::scaling_limits_dgff}. This is analog of~\cite[Section~4.5]{SchrammSheffield09}. 
Define a new time parameter, for $t\ge 0$,
\[s(t):=\log\left(\frac{V_t^R-V_t^L}{y^R-y^L}\right).\]
Set $\tilde{W}_s=\hat{W}_t$ when $s=s(t)$. Then we have 
\begin{equation}\label{eqn::SDE_changecoordinate}
d\tilde{W}_s=q(\tilde{W}_s)d\tilde{B}_s,\quad \text{where }q(x)=\sqrt{2(1-x^2)},
\end{equation}
and $(\tilde{B}_s, s\ge 0)$ is a standard one-dimensional Brownian motion. The process $\tilde{W}$ starts from $\tilde{W}_0=\frac{-y^L-y^R}{y^R-y^L}$, evolves according to~\eqref{eqn::SDE_changecoordinate}, and stops when it hits $\{-1, +1\}$ at finite time $T$. 
From above, we see how to transform from the process in~\eqref{eqn::SDE_sle4} to the one in~\eqref{eqn::SDE_changecoordinate}. 
We will show that this is one-to-one transform up to a scaling constant. 

\begin{lemma}\label{lem::coordinatechange}
Suppose $(\tilde{Y}_s)$ is a continuous process starting from $\tilde{Y}_0\in (-1, 1)$, evolving according to $d\tilde{Y}_s=q(\tilde{Y}_s)d\tilde{B}_s$, and stopped when it hits $\{-1, +1\}$. Define
\[t(s)=\frac{1}{8}\int_0^s e^{2u}(1-\tilde{Y}_u^2)du,\quad s(t)=\sup\{s: t(s)\le t\}.\]
Define 
\[Y_t=\frac{1}{2}e^{s(t)}\tilde{Y}_{s(t)}+\frac{1}{2}\int_0^{s(t)} e^u\tilde{Y}_udu-\frac{1}{2}\tilde{Y}_0. \]
Then $(Y_t)$ is a  continuous process starting from $0$ and evolving according to  
\begin{equation}\label{eqn::SDE_aux}
dY_t=2dB_t+\frac{2dt}{V_t^L-Y_t}+\frac{2dt}{V_t^R-Y_t}, \quad dV_t^L=\frac{2dt}{V_t^L-Y_t}, \quad dV_t^R=\frac{2dt}{V_t^R-Y_t},
\end{equation} 
where $V_0^L=-\frac{1}{2}(1+\tilde{Y}_0)$ and $V_0^R=\frac{1}{2}(1-\tilde{Y}_0)$. 
\end{lemma}
\begin{proof}
Set
\[B_t=\frac{1}{4}\int_0^{s(t)}e^uq(\tilde{Y}_u)d\tilde{B}_u;\]
\[V^L_t=V_0^L-\frac{1}{2}\int_0^{s(t)} e^u(1-\tilde{Y}_u)du,\quad V^R_t=V_0^R+\frac{1}{2}\int_0^{s(t)}e^u(1+\tilde{Y}_u)du.\]
Then $(B_t)$ is a standard Brownian motion, and $(Y_t, V_t^L, V_t^R)$ solves~\eqref{eqn::SDE_aux}. 
\end{proof}

\subsection{Scaling limits of level lines in discrete GFF}
\label{subsec::scaling_limits_dgff}
We will derive Theorem~\ref{thm::crossingproba_cvg_dgff} in this section following the proof in~\cite{SchrammSheffield09}. Although our setup is different from it is in~\cite{SchrammSheffield09}, similar techniques work. In fact, our setting is much easier to treat. As the proof in~\cite{SchrammSheffield09} is long and technical (and we do not know how to simplify it), we only sketch the proof in our setting. 

Recall that $\phi_{\delta}$ is discrete GFF on $V_{\delta}$ with alternating boundary condition~\eqref{eqn::alternatingbc}, and $\eta_{\delta}$ is the level line of $\phi_{\delta}$ starting from $b_{\delta}^{\diamond}$ and stopped when it hits the boundary segment $[c_{\delta}, a_{\delta})$. Fix a conformal map $\varphi$ from $R_L$ onto $\HH$ with $\varphi(a)<\varphi(b)=0<\varphi(c)<\varphi(d)=\infty$. Denote by $y^L=\varphi(a)<0$ and by $y^R=\varphi(c)>0$. Consider $\varphi(\eta_{\delta})$ in $\HH$ parameterized by the half-plane capacity. Denote by $W_t^{\delta}$ its driving function and by $g_t^{\delta}$ the corresponding family of conformal maps. For $\eps>0$, let $T_{\eps}^{\delta}$ be the first time that $\varphi(\eta_{\delta})$ gets within $\eps$-distance of the points $\varphi(a_{\delta})$ or $\varphi(c_{\delta})$. Recall that $(W_t)$ is the driving function of $\eta\sim\SLE_4(-2; -2)$ as in~\eqref{eqn::SDE_sle4}. Let $T_{\eps}$ be the first time that $\eta$ gets within $\eps$-distance of the points $y^L$ or $y^R$. 

\begin{lemma}\label{lem::crossingproba_cvg_drivingfunction}
For any fixed $\eps, \alpha>0$ small, there is $\delta_0=\delta_0(\eps, \alpha)>0$ such that, if $\delta\le\delta_0$, the interface $\eta_{\delta}$ can be coupled with $\eta$ so that 
\[\PP\left(\sup_{t\in [0, T^{\delta}_{\eps}\wedge T_{\eps}]}|W^{\delta}_t-W_t|\ge \alpha\right)\le \alpha. \] 
\end{lemma}

\begin{proof}
Denote by $\tilde{W}^{\delta}$ the coordinate change of $W^{\delta}$ as in Section~\ref{subsec::sle}. 
Let $\LF_s$ be the $\sigma$-field generated by $\sigma(\tilde{W}^{\delta}_r, r\le s)$. We will first prove the following conclusion which is analog of~\cite[Proposition~4.2]{SchrammSheffield09}. For any $\eps, \alpha, \beta>0$ small, there exists $C>0$ depending on $\eps$ and there exists $\delta_0>0$ depending on $\eps, \alpha, \beta$ such that the following holds. If $\delta\le\delta_0$ and $s_0, s_1$ are two stopping times for $\tilde{W}^{\delta}$ such that almost surely $s_0\le s_1\le T_{\eps}^{\delta}$, $s_1-s_0\le \beta^2$, and $\sup_{s\in [s_0, s_1]}|\tilde{W}^{\delta}_s-\tilde{W}^{\delta}_{s_0}|\le \beta$, then the following two estimates hold with probability at least $1-\alpha$:
\begin{equation}\label{eqn::keyestimates}
\left|\E\left[\Delta\tilde{W}^{\delta}\cond \LF_{s_0}\right]\right|\le C\beta^3,\quad \left|\E\left[(\Delta\tilde{W}^{\delta})^2-q(\tilde{W}^{\delta}_{s_0})^2\Delta s\cond \LF_{s_0}\right]\right|\le C\beta^3,
\end{equation}
where $\Delta s=s_1-s_0$ and $\Delta\tilde{W}^{\delta}=\tilde{W}^{\delta}_{s_1}-\tilde{W}^{\delta}_{s_0}$. Roughly speaking, \eqref{eqn::keyestimates} corresponds to the discrete version of~\eqref{eqn::SDE_changecoordinate}. 

For $t>0$, let $F^{\delta}_t$ be the function defined on $V_{\delta}$ that is $-\lambda_0$ on vertices to the right side of $\eta_{\delta}[0,t]$, 
$+\lambda_0$ on the vertices to the left side of $\eta_{\delta}[0,t]$, equal to the boundary data on $\partial V_{\delta}$, and is harmonic at all other vertices in $V_{\delta}$. Suppose $v\in V_{\delta}$ such that the distance between $v$ and $\partial R_L$ is at least $1/4$. For $k=0,1$, define $X_k=\E[\phi_{\delta}(v)\cond \LF_{s_k}]$ and let $\LA_k$ be the event 
\[|X_k-F_{s_k}(v)|\ge\beta^5. \]
By~\cite[Proposition~3.27]{SchrammSheffield09}, we have $\PP(\LA_k)\le \frac{1}{4}\alpha\beta^5$ if $\delta$ is small enough. Consequently, we have
\begin{align*}
|X_0-F_{s_0}(v)|&\le \beta^5,\quad\text{on }\LA_0^c;\\%\label{eqn::keyestimates_aux1}\\
|\E[X_1-F_{s_1}(v)\cond \LF_{s_0}]|&\le \beta^5+O(1)\PP(\LA_1\cond \LF_{s_0}). %\label{eqn::keyestimates_aux2}
\end{align*}
Let $\LA$ be the event $\PP(\LA_1\cond \LF_{s_0})\ge \beta^5$, then $\PP(\LA)\le \frac{1}{4}\alpha$. Combing the above two estimates and the fact that $\E[X_1\cond\LF_{s_0}]=X_0$, we have
\[\E[F_{s_1}(v)\cond\LF_{s_0}]-F_{s_0}(v)=O(\beta^5),\quad\text{on }\LA_0^c\cap\LA^c.\]
For $k=0,1$, let $H_k$ be the (continuous) function that is $-\lambda_0$ to the right side of $\eta_{\delta}[0,s_k]$, 
$+\lambda_0$ to the left side of $\eta_{\delta}[0,s_k]$, equal to the boundary data on $\partial R_L$, and is harmonic in $R_L\setminus\eta_{\delta}[0,s_k]$. Then $H_k(v)-F_{s_k}(v)$ is small if $\delta$ is small enough and the distance between $v$ and $\eta_{\delta}[0,s_k]$ is bounded from below. Precisely, let $\LB_k$ be the event $|H_k(v)-F_{s_k}(v)|\ge \beta^5$, we have $\PP(\LB_k)\le\frac{1}{4}\alpha\beta^5$ if $\delta$ is small enough. Let $\LB$ be the event $\PP(\LB_1\cond\LF_{s_0})\ge\beta^5$, then $\PP(\LB)\le\frac{1}{4}\alpha$. From the above estimate, we have
\begin{equation}\label{eqn::keyestimates_aux}
\E[H_1(v)\cond\LF_{s_0}]-H_0(v)=O(\beta^5),\quad\text{on }\LA_0^c\cap\LA^c\cap\LB_0^c\cap\LB^c.
\end{equation}
For $k=0,1$, denote by 
\[Z_k=\frac{2g^{\delta}_{s_k}(\varphi(v))-g^{\delta}_{s_k}(y^L)-g^{\delta}_{s_k}(y^R)}{g^{\delta}_{s_k}(y^R)-g^{\delta}_{s_k}(y^L)}. \]
Then we have
\[H_k(v)=\lambda_0-\frac{2\lambda_0}{\pi}\arg(Z_k-1)+\frac{2\lambda_0}{\pi}\arg(Z_k-\tilde{W}^{\delta}_{s_k})-\frac{2\lambda_0}{\pi}\arg(Z_k+1).\]
By the calculation in \cite[Proof of Proposition~4.2]{SchrammSheffield09}, we have
\[\frac{\pi}{2\lambda_0}(H_1(v)-H_0(v))=\frac{y}{x^2+y^2}\left(\frac{x}{x^2+y^2}\left(q(\tilde{W}^{\delta}_{s_0})^2\Delta s-(\Delta\tilde{W}^{\delta})^2\right)-\Delta\tilde{W}^{\delta}\right)+O(\beta^3),\]
where $x=\Re(Z_0-\tilde{W}^{\delta}_{s_0})$ and $y=\Im(Z_0)$. Plugging into~\eqref{eqn::keyestimates_aux}, with probability at least $1-\alpha$, we have
\[\E\left[\frac{x}{x^2+y^2}\left(q(\tilde{W}^{\delta}_{s_0})^2\Delta s-(\Delta\tilde{W}^{\delta})^2\right)-\Delta\tilde{W}^{\delta}\cond \LF_{s_0}\right]=O(\beta^3). \]
By different choice of $v$, we obtain~\eqref{eqn::keyestimates}. 

Recall that $W$ is driving function of $\SLE_4(-2;-2)$. Denote by $\tilde{W}$ the coordinate change of $W$ as in Section~\ref{subsec::sle}. 
With~\eqref{eqn::keyestimates} at hand, by arguments in \cite[Section~4.4]{SchrammSheffield09}, we have the following conclusion. For any fixed $\eps, \alpha>0$ small, there is $\delta_0>0$ such that, if $\delta\le\delta_0$, there is coupling between $\tilde{W}^{\delta}$ and $\tilde{W}$ so that 
\[\PP\left(\sup_{s\in [0, T_{\eps}^{\delta}\wedge T_{\eps}]}|\tilde{W}^{\delta}_s-\tilde{W}_s|\ge\alpha\right)\le \alpha. \]
Finally, by Lemma~\ref{lem::coordinatechange} and arguments in \cite[Section~4.6]{SchrammSheffield09}, we obtain the conclusion. 
\end{proof}

\begin{proof}[Proof of Theorem~\ref{thm::crossingproba_cvg_dgff}]
Recall that $\eta_{\delta}$ is the level line of $\phi_{\delta}$, we parameterize $\varphi(\eta_{\delta})$ by the half plane capacity, and we denote by $W^{\delta}_t$ its driving function. Recall that $\eta\sim\SLE_4(-2;-2)$ in $\HH$ with force points $y^L=\varphi(a)<0$ and $y^R=\varphi(c)>0$, we denote by $W_t$ its driving function.  
From Lemma~\ref{lem::crossingproba_cvg_drivingfunction}, $W^{\delta}$ is close to $W$ in local uniform topology. Combing with \cite[Section~4.7 and Lemma~4.16]{SchrammSheffield09}, $\varphi(\eta^{\delta})$ and $\eta$ are close in Hausdorff metric. Precisely, for any fixed $\eps, \alpha>0$ small, there is $\delta_0>0$ such that, if $\delta\le\delta_0$, the interface $\eta_{\delta}$ can be coupled with $\eta$ so that
\[\PP\left(\sup_{t\in [0, T_{\eps}^{\delta}\wedge T_{\eps}]}d_H(\varphi(\eta_{\delta}[0,t], \eta[0,t])\ge\alpha\right)\le\alpha,\]
where $d_H$ denotes the Hausdoff metric. For such argument to work, it is important that $(\eta(t), 0\le t\le T)$ hits $\R$ only at the two end points $\eta(0)=0$ and $\eta(T)\in \{y^L, y^R\}$. Then, using arguments in \cite[Section~4.8]{SchrammSheffield09}, we arrive at the following conclusion. For any fixed $\eps, \alpha>0$ small, there is $\delta_0>0$ such that, if $\delta\le\delta_0$, the interface $\eta_{\delta}$ can be coupled with $\eta$ so that 
\[\PP\left(\sup_{t\in [0, T_{\eps}^{\delta}\wedge T_{\eps}]}|\varphi(\eta_{\delta}(t)-\eta(t)|\ge\alpha\right)\le\alpha.\]

It remains to get the convergence of the whole process. Suppose $\eta_*$ is any subsequential limit of $\eta_{\delta}$ in Hausdorff metric. From the above argument, we see that $\varphi(\eta_*)$ has the same law as $\eta$ up to $T_{\eps}$ for any $\eps>0$. Note that $\eta$ is continuous up to and including $T$. We may conclude that $\varphi(\eta_*)$ has the same law as $\eta$ up to $T$. This gives the convergence of the whole process. Consequently, we obtain the convergence of the crossing probability (combining with~\eqref{eqn::sle_crossingproba})
\begin{align*}
\lim_{\delta\to 0}\mbb{P}\left([a_\delta, b_\delta) \stackrel{\phi_\delta \geq 0}{\longleftrightarrow}[c_\delta, d_\delta)\right)
=\PP(\eta(T)=y^R)=\frac{-y^L}{y^R-y^L}.
\end{align*}
\end{proof}
We end this section by a discussion on the convergence of discrete GFF level lines in Theorem~\ref{thm::crossingproba_cvg_dgff} when $\lambda\neq \lambda_0$. By analyzing level lines of continuous GFF as in \cite{WangWuLevellinesGFFI}, we believe that the level line $\eta_{\delta}$ of $\phi_{\delta}$ converges weakly to $\eta\sim \SLE_4(-2\rho, \rho-1; \rho-1, -2\rho)$ in $R_L$ from $b$ to $d$ with force points $(a, b^-; b^+, c)$ where $\rho=\lambda/\lambda_0$. However, the techniques in~\cite{SchrammSheffield09} do not  apply to this general setting directly to our knowledge. In particular, the authors in~\cite{SchrammSheffield09} derived the convergence of driving function when $\lambda=0$: they proved that the driving function of $\eta_{\delta}$ weakly converges to the driving function of $\SLE_4(-1; -1)$ in the local uniform topology; however, the convergence in stronger topology is still missing. Assuming the convergence of $\eta_{\delta}$ to $\eta\sim\SLE_4(-1;-1)$ in Haudorff metric, we may conclude that the crossing probability in~\eqref{DGFF zero boundary result} is convergent:
\[\lim_{\delta\to 0}\mbb{P}\left([a_\delta, b_\delta) \stackrel{\phi_\delta \geq 0}{\longleftrightarrow}[c_\delta, d_\delta)\right)=\PP(\eta\text{ hits }[c,d) \text{ before }[d,a)).\]
Whereas, to get an explicit formula as in~\eqref{eqn::crossingproba_limit} for the right-hand side is another open question. 

\subsection{Scaling limits of level lines in metric graph GFF}
\label{subsec::scaling_limits_mgff}

Recall that $\tilde{\phi}_{\delta}$ is metric graph GFF on $\tilde{V}_{\delta}$. Suppose the boundary condition is the following: (note that this is different from the one in~\eqref{eqn::alternatingbc})
\begin{equation}\label{eqn::mgff_bc_aux}
\tilde{\phi}_{\delta}(v)=
\begin{cases}
\lambda, & v\in [a_{\delta},b_{\delta})\cup [c_{\delta},d_{\delta}),\\
0,& v\in [b_{\delta},c_{\delta})\cup [d_{\delta},a_{\delta}).
\end{cases}
\end{equation} 
Recall that 
\[\tilde{\A}_{\delta, 0}^l=\left\{v\in\tilde{V}_{\delta}: \exists \text{ continuous path }\gamma\text{ connecting $v$ to $[a_{\delta},b_{\delta})$ such that }\tilde{\phi}_{\delta}\ge 0 \text{ on }\gamma\right\}. \]
Define the frontier of $\tilde{\A}_{\delta, 0}^l$ as follows: Consider the set of all the points in $\partial\tilde{\A}_{\delta, 0}^l$ that are connected in $\tilde{\A}_{\delta, 0}^l$ to $[a_{\delta},b_{\delta})$ and in $\tilde{V}_{\delta}\setminus\tilde{\A}_{\delta, 0}^l$ to $[b_{\delta},c_{\delta})$. This set contains no vertices and the edges it intersects in the dual graph give a path from the point $b_{\delta}^{\diamond}$ to the point $(3\delta/2, 1)$ (near $a$) or the point $(L, 3\delta/2)$ (near $c$). We call such path the frontier of $\tilde{\A}_{\delta, 0}^l$. Using conclusions in~\cite[Section~5.2]{ALS18}, we may conclude that, when $\lambda=2\lambda_0$, the frontier of $\tilde{\A}_{\delta, 0}^l$ weakly converges to $\SLE_4(-2;-2)$ in $R_L$ from $b$ to $d$ with force points $(a;c)$ in Hausdorff metric as $\delta=2^{-n}\to 0$. Consequently, we have the convergence of the crossing probabilities: (see~\cite[Remark~5.11]{ALS18})
\begin{equation}\label{eqn::crossingproba_limit_mgff}
\lim_{\delta=2^{-n}\to 0}\mbb{P}\left([a_\delta, b_\delta) \stackrel{\tilde{\phi}_\delta \geq 0}{\longleftrightarrow}[c_\delta, d_\delta)\right)=
\frac{(\varphi(b)-\varphi(a))(\varphi(d)-\varphi(c))}{(\varphi(c)-\varphi(a))(\varphi(d)-\varphi(b))},
\end{equation}
where $\varphi$ is any conformal map from $R_L$ onto the upper-half plane $\HH$ with $\varphi(a)<\varphi(b)<\varphi(c)<\varphi(d)$. We emphasize that the metric graph GFF~\eqref{eqn::crossingproba_limit_mgff} has the boundary condition~\eqref{eqn::mgff_bc_aux} with $\lambda=2\lambda_0$ which is different from the one in~\eqref{eqn::alternatingbc}. 

Let us go back to the boundary condition~\eqref{eqn::alternatingbc} and to~\eqref{MGFF alternating boundary result} in Theorem~\ref{thm::Alternating boundary result}. Suppose we have metric graph GFF with boundary condition~\eqref{eqn::alternatingbc}. Using conclusions in~\cite{ALS18}, we may conclude that the crossing probabilities in~\eqref{MGFF alternating boundary result} are convergent. By Theorem~\ref{thm::Alternating boundary result}, we know that the limit should be different from the one in the case of discrete GFF. To derive the explicit formula for this limit is an interesting question. We will give this explicit formula when $\lambda=2\lambda_0$ in a forthcoming paper~\cite{LW20}.

%%%%%%%%%%%%%%%%%%%%%%%%%%%%

\end{document}